\documentclass[11pt,a4paper,DIV=10]{scrartcl}
\author{%
Jan Corsten\thanks{London School of Economics, Houghton St, London WC2A 2AE, UK, \href{mailto:j.corsten@lse.ac.uk}{j.corsten@lse.ac.uk}. Supported by an LSE studentship.}%
\and Walner Mendon\c{c}a\thanks{University of Warwick, Mathematics Institute, Coventry CV4 7AL, UK, \href{mailto:walner@impa.br}{walner@impa.br}. Supported by CAPES project 88882.332408/2010-01 and by the UK Research and Innovation Future Leaders Fellowship MR/S016325/1}
}%
\date{}
\title{\LARGE{Tiling edge-coloured graphs with few monochromatic bounded-degree graphs}}

\usepackage[utf8]{inputenc}
\usepackage[T1]{fontenc}
\usepackage[british]{babel}
\usepackage{lmodern}
\usepackage{amsmath,amsthm,amsfonts,amssymb}
\usepackage[shortlabels]{enumitem}

\usepackage{tikz}
\usetikzlibrary{calc,decorations.pathmorphing,shapes.geometric}
\pgfdeclarelayer{background}
\pgfdeclarelayer{foreground}
\pgfdeclarelayer{front}
\pgfsetlayers{background,main,foreground,front}

\usepackage[colorlinks=true, linkcolor=gray, filecolor=magenta, urlcolor=cyan, citecolor=gray, unicode]{hyperref}
\usepackage[capitalise,noabbrev]{cleveref}
\crefname{chapter}{Chapter}{Chapters}
\crefname{section}{Section}{Sections}
\crefname{subsection}{Section}{Sections}
\crefname{subsubsection}{Section}{Sections}
\crefname{figure}{Figure}{Figures}
\crefname{table}{Table}{Tables}
\crefname{conj}{Conjecture}{Conjectures}
\crefname{claim}{Claim}{Claims}
\crefname{equation}{}{}
\crefname{enumi}{}{}

\def\claimqed{\scalebox{.6}{$\Box$}}
\newenvironment{claimproof}[1][Proof]{
  \renewcommand{\qedsymbol}{\claimqed}
  \begin{proof}[#1]
  }{
  \end{proof}
}

\theoremstyle{definition}
\newtheorem{innercustomgeneric}{\customgenericname}
\providecommand{\customgenericname}{}
\newcommand{\newcustomdefinition}[2]{%
  \newenvironment{#1}[1]
  {%
   \renewcommand\customgenericname{#2}%
   \renewcommand\theinnercustomgeneric{##1}%
   \innercustomgeneric{}
  }
  {\endinnercustomgeneric}
}

\numberwithin{equation}{section}

\theoremstyle{definition}
\newtheorem{question}{Question}
\newtheorem{definition}{Definition}[section]
\newtheorem{example}[definition]{Example}

\newcustomdefinition{step}{Step}

\theoremstyle{plain}
\newtheorem{thm}[definition]{Theorem}
\newtheorem{conj}{Conjecture}[section]
\newtheorem{lemma}[definition]{Lemma}
\newtheorem{cor}[definition]{Corollary}
\newtheorem{prop}[definition]{Proposition}

\newtheorem{claim}{Claim}[subsection]

\theoremstyle{remark}

\newcommand{\Pro}[1]{\mathbb{P} \left[#1\right]}

\newcommand{\expb}[1]{\exp\left(#1\right)}

\newcommand{\bN}{\mathbb{N}}

\newcommand{\e}{\varepsilon}
\newcommand{\eps}{\varepsilon}
\let\epsilon=\varepsilon{}

\newcommand{\cF}{\mathcal{F}}

\newcommand{\cH}{\mathcal{H}}
\newcommand{\cI}{\mathcal{I}}

\newcommand{\cK}{\mathcal{K}}

\newcommand{\cT}{\mathcal{T}}

\newcommand{\card}[1]{\left| #1 \right|}
\renewcommand{\subset}{\subseteq}

\DeclareMathOperator{\dd}{d}

\begin{document}
\maketitle
\begin{abstract}
We prove that for all integers $\Delta,r \geq 2$, there is a constant $C = C(\Delta,r) >0$ such that the following is true for every sequence $\cF = \{F_1, F_2, \ldots\}$ of graphs with $v(F_n) = n$ and $\Delta(F_n) \leq \Delta$, for each $n \in \bN$. In every $r$-edge-coloured $K_n$, there is a collection of at most $C$ monochromatic copies from $\cF$ whose vertex-sets partition $V(K_n)$. This makes progress on a conjecture of Grinshpun and S\'ark\"ozy.

\par\vskip\baselineskip\noindent
\scriptsize{\textbf{2020 Mathematics Subject Classification: 05C55 (primary); 05C70 (secondary).}}
\end{abstract}

\section{Introduction and main results}

A conjecture of Lehel states that the vertices of any $2$-edge-coloured complete graph can be
partitioned into two monochromatic cycles of different colours. Here, single vertices and edges are
considered cycles. This conjecture first appeared in~\cite{Ayel1979}, where it was also proved for some
special types of colourings of $K_n$. \L{}uczak, R\"odl and Szemer\'edi~\cite{Luczak1998} proved Lehel's
conjecture for sufficiently large $n$ using the regularity method. Allen~\cite{Allen2008} gave an
alternative proof, with a better bound on $n$. Finally, Bessy and Thomass\'e~\cite{Bessy2010} proved
Lehel's conjecture for all integers $n\ge 1$.

For colourings with more colours, Erd\H os, Gy\'arf\'as and Pyber~\cite{Erdos1991} proved that the
vertices of every $r$-edge-coloured complete graph on $n$ vertices can be partitioned into $O(r^2
\log r)$ monochromatic cycles. They further conjectured that $r$ cycles should be enough. The
currently best-known upper bound is due to Gy\'arf\'as, Ruszink\'o, S\'ark\"ozy and
Szemer\'edi~\cite{Gyarfas2006}, who showed that $O(r \log r)$ cycles suffice. However, the conjecture was
refuted by Pokrovskiy~\cite{Pokrovskiy2014}, who showed that, for every $r \geq 3$, there exist infinitely
many $r$-edge-coloured complete graphs which cannot be vertex-partitioned into $r$ monochromatic
cycles. Nevertheless, Pokrovskiy conjectured that in every $r$-edge-coloured complete graph one can
find $r$ vertex-disjoint monochromatic cycles which cover all but at most $c_r$ vertices for some
$c_r \geq 1$ only depending on $r$ (in his counterexample $c_r = 1$ is possible).

In this paper, we study similar problems in which we are given a family of graphs $\cF$ and an
edge-coloured complete graph $K_n$ and our goal is to partition $V(K_n)$ into monochromatic copies
of graphs from $\cF$. All families of graphs $\cF$ we consider here are of the form $\cF = \{F_1,
F_2, \ldots\}$, where $F_i$ is a graph on $i$ vertices for every $ i \in \bN$. We call such a family
a \emph{sequence of graphs}. A collection $\cH$ of vertex-disjoint subgraphs of a graph $G$ is an
\emph{$\cF$-tiling} of $G$ if $\cH$ consists of copies of graphs from $\cF$ with $V(G) = \bigcup_{H
  \in \cH} V(H)$. If $G$ is edge-coloured, we say that $\cH$ is \emph{monochromatic} if every $H \in
\cH$ is monochromatic. Let $\tau_r(\cF,n)$ be the minimum $ t \in \bN$ such that for every
$r$-edge-coloured $K_n$, there is a monochromatic $\cF$-tiling of size at most $t$. We define the
\emph{tiling number} of $\cF$ as
\[
  \tau_r(\cF) = \sup_{n \in \bN} \tau_r(\cF,n).
\]

Using this notation, the results of Pokrovskiy~\cite{Pokrovskiy2014} and of Gy\'arf\'as, Ruszink\'o,
S\'ark\"ozy and Szemer\'edi~\cite{Gyarfas2006} mentioned above imply that $r + 1 \leq
\tau_r(\cF_{\mathrm{cycles}}) = O(r \log r)$, where $\cF_{\mathrm{cycles}}$ is the family of cycles.
Note that, in general, it is not clear at all that $\tau_r(\cF)$ is finite and it is a natural
question to ask for which families this is the case.

The study of such tiling problems for general families of graphs was initiated by Grinshpun and
S\'ark\"ozy~\cite{Grinshpun2016}. The \emph{maximum degree} $\Delta(\cF)$ of a sequence of graphs
$\cF$ is given by $\sup_{F \in \cF} \Delta(F)$, where $\Delta(F)$ is the maximum degree of $F$. We
denote by ${\cF}_{\Delta}$ the collection of all sequences of graphs $\cF$ with $\Delta(\cF) \leq
\Delta$. Grinshpun and S\'ark\"ozy proved that \(\tau_2(\cF) \leq 2^{O(\Delta \log
  \Delta)}\) for all $\cF \in {\cF}_{\Delta}$. In particular, $\tau_2(\cF)$ is finite whenever
$\Delta(\cF)$ is finite. They also proved that $\tau_2(\cF) \leq 2^{O(\Delta)}$ for every sequence
of bipartite graphs $\cF$ of maximum degree at most $\Delta$, and showed that this is best possible
up to a constant factor in the exponent (see also~\cref{sec:concluding-remarks} for a more detailed discussion on the lower bound).

S\'ark\"ozy~\cite{Sarkoezy2017} further proved that $\tau_2(\cF_{\text{$k$-cycles}}) = O(k^2\log{k})$,
where $\cF_{\text{$k$-cycles}}$ denotes the family of $k$th power of cycles\footnote{The $k$-th
  power of a graph $H$ is the graph obtained from $H$ by adding an edge between any two vertices at
  distance at most $k$}. For more than two colours less is known. Answering a question of Elekes,
Soukup, Soukup and Szentmikl\'ossy~\cite{Elekes2017}, Bustamante, Frankl, Pokrovskiy, Skokan and the
first author~\cite{Bustamante2019} proved that $\tau_r(\cF_{\text{$k$-cycles}})$ is finite for all
$r,k \in \bN$. Grinshpun and S\'ark\"ozy~\cite{Grinshpun2016} conjectured that the same should be
true for all families of graphs of bounded degree with an exponential bound.

\begin{conj}[Grinshpun-S\'ark\"ozy~\cite{Grinshpun2016}, 2016]\label{conj:main}
  For every $r,\Delta \in \bN$ and $\mathcal{F}\in\mathcal{F}_{\Delta}$, $\tau_r(\cF)$ is finite.
  Moreover, there is some $C_r >0 $ such that $ \tau_r(\cF) \leq \exp(\Delta^{C_r})$.
\end{conj}

Our main theorem shows that $\tau_r(\cF)$ is indeed finite. For a given positive integer $k$, we
denote by $\exp^k$ the $k$th-composition of the exponential function.

\begin{thm}\label{thm:main}
  There is an absolute constant $K > 0$ such that for all integers $r,\Delta \geq 2$ and all $\cF
  \in {\cF}_{\Delta}$, we have
  \[
    \tau_r(\cF) \leq \exp^2\left(r^{K r \Delta^3} \right).
  \]
  In particular, $\tau_r(\cF)$ is finite whenever $\Delta(\cF)$ is finite.
\end{thm}

In order to prove \cref{thm:main}, we shall prove an absorption lemma (see \cref{lem:abs}) whose proof relies on a \emph{density increment argument}. This is responsible for the double
exponential bound in our main theorem.

The paper is organized as follows. In \cref{sec:overview}, we present an overview of the proof
of our main theorem and the proof of our absorption lemma. In \cref{sec:reg} we collect a few lemmas regarding regular pairs and regular cylinders that we shall use repeatedly in later sections.
The proof of our absorption lemma and main theorem can be found in \cref{subsec:abs} and \cref{subsec:main}, respectively. Finally, we finish the paper with some concluding remarks in
\cref{sec:concluding-remarks}.

\section{Proof overview}\label{sec:overview}
The proof of \cref{thm:main}, similarly to the proof of the two colour result of Grinshpun and Sárközy~\cite{Grinshpun2016}, combines ideas from the absorption method as in the original paper of Erd\H os, Gy\'arf\'as and Pyber~\cite{Erdos1991} with some modern approaches involving the blow-up lemma and the weak regularity lemma of Duke, Lefmann and R\"odl~\cite{Duke1995}. However, in order to extend these ideas to more colours, we need to prove a significantly more complicated \emph{absorption lemma}, requiring new ideas involving a density increment argument.

Our absorption lemma (\cref{lem:abs}) states that if we have
$k:=\Delta+2$ disjoint sets of vertices $V_1,\ldots,V_k$ with $|V_i| \geq 2|V_1|$ for all $i=2,\ldots,k$ such that every vertex in $V_1$ belongs to at least $\delta |V_2|\cdots |V_k|$ monochromatic $k$-cliques \emph{transversal}\footnote{A $k$-clique is transversal in $(V_1,\ldots,V_k)$ if it contains one vertex in each one of the sets $V_1,\ldots,V_k$.}
in $(V_1,\ldots,V_k)$, then it is possible to cover the vertices in $V_1$ with a constant number (depending on $\delta$, $r$ and $\Delta$) of
monochromatic vertex disjoint copies of graphs from $\mathcal{F}$. Furthermore, we can choose such a covering using no more than $|V_1|$ vertices in each $V_2,\ldots,V_k$.

To deduce \cref{thm:main} from the absorption lemma, we need to partition $V(K_n)$ in a similar fashion as in~\cite{Bustamante2019}: first we find $k-1$ monochromatic
\emph{super-regular cylinders} $Z_1,\ldots,Z_{k-1}$ covering a positive proportion of the vertices
of $K_n$ (see \cref{sec:reg} for the definition of super-regular cylinders). Then we apply a result of Fox and
Sudakov~\cite{Fox2009} to \emph{greedily} cover with few disjoint monochromatic copies of
graphs from $\mathcal{F}$ almost all of the vertices in
$V(K_n) \setminus (Z_1 \cup \cdots \cup Z_{k-1})$, leaving uncovered a set $R$ of size much smaller than $|Z_{k-1}|$ (see~\cref{prop:greedy}).

Now we split $R$ into two sets: the set $R_1$ of vertices belonging to at least
$\delta |Z_1| \cdots |Z_{k-1}|$ monochromatic $k$-cliques transversal in $(R,Z_1,\ldots,Z_{k-1})$,
and the set $R_2 = R\setminus R_1$. Using our absorption lemma we can cover the vertices in $R_1$ using
no more than $|R_1|$ vertices of each of the cylinders $Z_1,\ldots,Z_{k-1}$. For each
$i=1,\ldots,k-1$, let $Z'_i$ be the set of vertices in $Z_i$ that has not been used to cover $R_1$.
Since we $|R_1|$ is significantly smaller than $|Z_i|$, it follows that each $Z'_i$ is still a super-regular cylinder.
Now, if the set $R_2$ was empty, then we would be done. Indeed, a consequence of the blow-up lemma (\cref{thm:blowup}) guarantees that we can cover each of the cylinders $Z'_1,\ldots,Z'_{k-1}$ with $k+1$ copies of vertex disjoint monochromatic graphs from $\mathcal{F}$.

So let us consider the case where $R_2$ is non-empty.
In this case, we repeat the process above.
This time we first find a reasonably large regular cylinder $Z_k$ in $R_2$, then we greedily cover most of the vertices in $R_2\setminus Z_k$ and apply the absorption lemma to those vertices that have not yet been covered and belong to many monochromatic $k$-cliques transversal in $R_2$ and $k-1$ of the cylinders $Z'_1,\ldots,Z'_{k-1},Z_k$.
The set of leftover vertices, which we denote by $R_3$, is either empty (and in this case we are done, as above) or is non-empty, in which case we repeat the process to cover $R_3$.
Finally, using a lemma from~\cite{Bustamante2019} (see~\cref{lem:indtr}) and Ramsey's theorem, we can show that this process must stop after $R_r(K_k)$ many iterations, where $R_r(K_k)$ denotes the $r$-colour Ramsey number of the graph $K_k$.

In order to prove the absorption lemma, we employ a density increment argument.
This is the most difficult part of the proof and the key new idea in this paper.
First, we partition $V_1$ into $r$ sets $V_1^{(1)}, \ldots, V_1^{(r)}$ so that for every $j \in [r]$, every $v \in V_i^{(j)}$ is incident to at least $d/r \cdot |V_2|\cdots|V_k|$ monochromatic cliques of colour $j$ which are transversal in $(V_1, \ldots, V_k)$. We will cover each of these sets separately, making sure not to repeat vertices.
Let us illustrate how to cover $V_1^{(1)}$.

We start by finding a large $k$-cylinder $Z=(U_1,\ldots,U_k)$ with $U_1 \subset V_1^{(1)}, U_2 \subset V_2, \ldots U_k \subset V_k$ which is super-regular in colour $1$.
We shall use $Z$ as an \emph{absorber} at the end of the proof to cover any small set of leftovers.
Next, we greedily cover most of $V_1^{(1)} \setminus U_1$ by monochromatic copies of $\cF$ until the set of uncovered vertices $R$ has size much smaller then $|U_1|$.
To cover the set $R$, we will find a partition
$R = S \cup T_2 \cup \ldots \cup T_k$,
where each vertex in $S$ belongs to many monochromatic $k$-cliques of colour $1$ which are transversal in $(S,U_2,\ldots,U_k)$ (allowing $S$ to be absorbed into the cylinder $Z$ at the end of the proof)
and each vertex in $T_i$, for $i \in \{2,\ldots,k\}$, belongs to at least $(\delta+\eta)
|V_2|\cdots|V_{i-1}| |U_i| \cdots |U_k|$ monochromatic $k$-cliques transversal in
$(T_i,V_2,\ldots,V_i,U_{i+1},\ldots,U_k)$, for some $\eta \ll \delta$.

To cover the vertices in each $T_i$, with $i \in \{2,\ldots,k\}$, we repeat the argument with
$(V_1,\ldots,V_k)$ replaced by $(T_i,V_2,\ldots,V_i,U_{i+1},\ldots,U_k)$ and $\delta$ replaced by $\delta+\eta$.
This is our density increment argument.
Since every time we repeat the argument we significantly increase the density of $k$-cliques, we can bound the number of required repetitions in terms of the initial density of $k$-cliques. 

While covering each of the sets $T_2,\ldots,T_k$, we shall guarantee that the set of vertices $X_i \subseteq U_i$ that we use to cover them has size much smaller than $|U_i|$ for all $i=2,\ldots,k$. This way, the cylinder $Z' = (U_1 \cup S,U_2\setminus X_2,\ldots,U_k \setminus X_k)$ will be super-regular in colour $1$ and thus we can cover $Z'$ using the blow-up lemma.
Repeating this for every colour $j\in [r]$, we get a covering of $V_1$ with $O_{\delta,r,\Delta}(1)$ many monochromatic disjoint copies of graphs from $\mathcal{F}$.

\section{Regularity}\label{sec:reg}
In this section, we will gather all the notations and results related to the classical regularity
technique which we require for the proof. We start by introducing some relevant notations. Let $G =
(V_1,V_2,E)$ be a bipartite graph with parts $V_1$ and $V_2$. For any $U_i \subseteq V_i$, $i=1,2$,
the density of the pair $(U_1,U_2)$ in $G$ is given by
  \[
    d(U_1,U_2) = \frac{e(U_1,U_2)}{|U_1||U_2|}.
  \]
  We say that $G$ (or the pair $(V_1,V_2)$) is $\e$-\emph{regular} if for all $U_i \subseteq V_i$ with $ |U_i| \geq \e |V_i|$, $i = 1,2$, we have
  \begin{equation*}
    \left| d(U_1,U_2)- d(V_1,V_2)\right| \leq \e.
  \end{equation*}
  If additionally we have $d(V_1,V_2) \geq d$ and $\deg(v,V_i) \geq \delta |V_i|$ for all $v \in
  V_{3-i}$, $i=1,2$, then we say that $G$ (or $(V_1,V_2)$) is $(\e,d,\delta)$-\emph{super-regular}.
  We often say that $G$ is \emph{$(\epsilon,d)$-super-regular} instead of
  $(\epsilon,d,d)$-super-regular.

  We begin with some simple facts about super-regular pairs. The first one is known as the slicing
  lemma and roughly says that if we take a large induced subgraph in a dense regular pair we still
  get a dense regular pair. Its proof is straightforward from the definition of a regular pair.
\begin{lemma}[Slicing lemma]\label{lem:slicing}
  Let $\beta > \e > 0$, $ d \in [0,1]$ and let $(V_1,V_2)$ be an $(\epsilon,d,0)$-super-regular pair. Then any pair $(U_1,U_2)$ with $|U_i| \geq \beta |V_i|$ and $U_i\subseteq V_i$, $i=1,2$, is $(\epsilon',d',0)$-super-regular with $\epsilon' = \max\{ \epsilon/\beta, 2\epsilon \}$ and $d' = d-\epsilon$.
\end{lemma}
The following lemma essentially says that after removing few vertices from a super-regular pair and
adding few new vertices with large degree, we still have a super-regular pair. The reader can find a
proof of it in \cref{appendix}.
\begin{lemma}\label{lem:robust}
  Let $0 < \e < 1/2$ and let $d,\delta \in [0,1]$ so that $\delta \geq 4 \e$. Let $(V_1,V_2)$ be an
  $(\epsilon,d,\delta)$-super-regular pair in a graph $G$. Let $X_i \subseteq V_i$ for $i \in \{1,2\}$,
  and let $Y_1,Y_2$ be disjoint subsets of $V(G) \setminus (V_1 \cup V_2)$. Suppose that for each $i
  \in \{1,2\}$ we have $|X_i|,|Y_i| \leq \e^2 |V_i|$ and $\deg(v,V_i) \geq \delta |V_i|$ for every $v
  \in Y_{3-i}$. Then the pair $( (V_1 \setminus X_1) \cup Y_1, (V_2 \setminus
  X_2) \cup Y_2)$ is $(8\epsilon,d-8\e,\delta/2)$-super-regular.
\end{lemma}

Let $k \geq 2$ be an integer and let $G$ be a graph. Given disjoint sets of vertices $V_1, \ldots,
V_k \subseteq V(G)$, we call $Z = (V_1, \ldots, V_k)$ a \emph{$k$-cylinder} and often identify it with the
induced $k$-partite subgraph $G[V_1, \ldots, V_k]$. We write $V_i(Z) = V_i$ for every $ i \in [k] $.
We say that $Z$ is \emph{$\e$-balanced} if
\[
  \max_{i\in[k]} \card{V_i(Z)} \leq (1+\eps) \min_{i\in[k]} \card{V_i(Z)}
\]
and \emph{balanced} if it is $0$-balanced. Furthermore, we say that $Z$ is $\epsilon$-regular if all
the $\binom{k}{2}$ pairs $(V_i,V_j)$ are $\epsilon$-regular. If $G$ is an $r$-edge-coloured graph and $i
\in [r]$, we say that $Z$ is $\e$-regular in colour $i$ if it is $\e$-regular in $G_i$, the graph
consisting of all edges of $G$ with colour $i$. Similarly, we define $(\eps,d)$-regular and
$(\epsilon,d,\delta)$-super-regular cylinders.

As sketched in \cref{sec:overview}, we will use super-regular cylinders as absorbers. The following lemma, which Grinshpun and S\'ark\"ozy~\cite{Grinshpun2016} deduced from the blow-up lemma~\cite{Komlos1997, Komlos1998, Sarkozy2014} and the Hajnal-Szemer\'edi theorem~\cite{Hajnal1970},\footnote{The second part of the theorem is not explicitly stated in \cite{Grinshpun2016} but follows readily from the blow-up lemma and the Hajnal-Szemer\'edi theorem.}
allows us to do this.

\begin{lemma}\label{thm:blowup}
  There is a constant $K$, such that for all $ 0 \leq \delta \leq d \leq 1/2$, $ \Delta \in \bN $, $k = \Delta+2$, $0 < \e \leq {( \delta d^\Delta )}^{K}$, and $ \cF \in {\cF}_{\Delta}$, the following is true for every $(\e,d,\delta)$-super-regular $k$-cylinder $Z = (V_1, \ldots, V_k)$.
  \begin{enumerate}[(i)]
    \item If $Z$ is $\e$-balanced, then its vertices can be partitioned into at most $\Delta+3$ copies of graphs from $\cF$.
    \item If $|V_i| \geq |V_1|$ for all $i = 2, \ldots, k$, then there is a copy of a graph from $\cF$ covering $V_1$ and at most $|V_1|$ vertices of each of $V_2, \ldots, V_k$.
  \end{enumerate}
\end{lemma}

It is important in the proof of \cref{thm:main} that we can find super-regular $k$-cylinders which cover linearly many vertices. The existence of such a pair follows readily from the regularity lemma.
Conlon and Fox~\cite[Lemma~5.3]{Conlon2012} used the weak regularity lemma of Duke, Lefmann, and R\"odl~\cite{Duke1995} to obtain better constants. We shall use the following coloured version of their result, the proof of which is very similar and can be found in \cref{appendix}. See
also~\cite[Lemma~2]{Grinshpun2016} for a 2-coloured version which follows readily from the
non-coloured version.

\begin{lemma}\label{lemma:regcyl}
  Let $k,r \geq 2$, $0 < \e < 1/(rk)$ and $ \gamma = \e^{r^{8rk}\e^{-5}}$. Then every $r$-edge-coloured complete graph on $n \geq 1/\gamma$ vertices contains, in one of the colours, a balanced
  $(\e,1/2r)$-super-regular $k$-cylinder $Z = (U_1, \ldots, U_k)$ with parts of size at least $\gamma
  n$.
\end{lemma}

The following lemma further guarantees that this remains possible as long as the host-graph has many $k$-cliques. It is also a straightforward consequence of the weak regularity lemma of Duke, Lefmann, and R\"odl and we provide a proof in \cref{appendix}.

\begin{lemma}\label{lem:regcyl-partite}
Let $k \geq 2$, and let $0 < \e < 1/2$ and $ 2k\e \leq d \leq 1$. Let $\gamma = \e^{k^2\e^{-12}}$.
Suppose that $G$ is a $k$-partite graph with parts $V_1, \ldots, V_k$ with at least $d|V_1| \cdots |V_k|$ cliques of size $k$.
Then there is some $\gamma' \in  [\gamma,\e]$ and an $(\e,d/2)$-super-regular $k$-cylinder $Z = (U_1, \ldots, U_k)$ in $G$ with $U_i \subset V_i$ and $|U_i| = \lfloor \gamma' |V_i| \rfloor$ for every $i \in [k]$.
\end{lemma}

\section{Proof of \texorpdfstring{\cref{thm:main}}{the main theorem}}\label{sec:proof}
In the proof, we will use the following theorem of Fox and Sudakov~\cite{Fox2009} about $r$-colour Ramsey numbers of bounded-degree graphs.

\begin{thm}[{\cite[Theorem 4.3]{Fox2009}}]\label{thm:bounded-ramsey}
  Let $k,\Delta,r,n \in \bN$ with $ r \geq 2$ and let $H_1, \ldots, H_r$ be $k$-partite graphs with
  $n$ vertices and maximum degree at most $\Delta$. Then
  \[
    R(H_1, \ldots, H_r) \leq r^{2rk \Delta} n.
  \]
\end{thm}

Recall that $\cF_{\Delta}$ denotes the collection of all sequences of graphs $\cF$ with $\Delta(F)
\leq \Delta$, for every $F \in \cF$, and let $\cF_{\Delta,k}$ be the collection of sequences $\cF
\in \cF_{\Delta}$ such that $F$ is $k$-partite, for every $F \in \cF$. Note that
$\cF_{\Delta}=\cF_{\Delta,\Delta+1}$. The following consequence of the previous theorem states
that, for each $\cF \in \cF_{k,\Delta}$, we can cover almost all vertices of $K_n$ with monochromatic copies of graphs from $\cF$.

\begin{prop}\label{prop:greedy}
  Let $\Delta,k,r \in \bN$, let $\gamma \in (0,1]$ and let $C = 4r^{2rk \Delta} \log (1/\gamma)$.
  Then, for every $ \cF \in \cF_{\Delta,k}$ and every $r$-edge-coloured $K_n$ with $n \geq r^{-2rk\Delta}$, it is possible to cover all but $\gamma n$ vertices of $K_n$ with at most $C$ vertex-disjoint monochromatic copies of graphs from $\cF$.
\end{prop}
\begin{proof}
  Let $ \cF = \left\{ F_1, F_2, \ldots \right\} \in \cF_{\Delta,k}$, $t = r^{-2rk \Delta}$,
  $C=(4/t)\log(1/\gamma)$ and $n \geq r^{-2rk\Delta}$.
  Consider $n_1 = \lfloor t n \rfloor \geq tn/2$. By \cref{thm:bounded-ramsey}, since $R_r(F_{n_1}) \leq
  t^{-1}n_1 \leq n$, there is a monochromatic copy of $F_{n_1}$ in $K_n$. Let $H_1$ be such copy and
  let $V_1 = V \setminus V(H_1)$. Note that $|V_1| = n - n_1 \leq (1-t/2)n$.

  Suppose that we have inductively found vertex-disjoint monochromatic graphs $H_1,\ldots,H_i \subseteq K_n$ that are copies of graphs in $\cF$ and such that $V_i := V(K_n) \setminus \left( V(H_1) \cup \cdots \cup V(H_i) \right)$ has at most ${(1-t/2)}^i n$ vertices.
  If $|V_i| \leq 2/t$, then we cover the vertices in $V_i$ with single vertices and stop the process.
  Therefore, suppose that $|V_i| \geq 2/t$.
  Then let $n_{i+1} = \lfloor t|V_i| \rfloor \geq t|V_i|/2$.
  Again by \cref{thm:bounded-ramsey}, since $R_r(F_{n_{i+1}}) \leq t^{-1}n_{i+1} \leq |V_i|$, there is a monochromatic copy of $F_{n_{i+1}}$ contained in $V_i$.
  Let $H_{i+1}$ be such a copy.
  Thus the set $V_{i+1} := V(K_n) \setminus \left( V(H_1) \cup \cdots \cup V(H_{i+1})\right)$ has size
  \begin{align*}
    |V_{{i+1}}| = |V_i| - n_{i+1} \leq (1-t/2)|V_i| \leq {(1-t/2)}^{i+1}n .
  \end{align*}

  Now, after $C/2$ steps, we have covered all but at most 
  \begin{align*}
    {(1-t/2)}^{C/2}n \leq e^{-(t/4)C}n \leq \gamma n
  \end{align*}
  vertices of $K_n$ using at most $C/2 + 2/t \leq C$ vertex-disjoint monochromatic copies of graphs from $\cF$.
\end{proof}

In particular, by choosing $\gamma = 1/n$, we get the following corollary.

\begin{cor}\label{cor:greedy}
  Let $\Delta,k,r \in \bN$ and let $C = 4r^{2rk \Delta} \log n$. Then, for every $ \cF \in
  \cF_{\Delta,k}$ and every $r$-edge-coloured $K_n$, there is a collection of at most $C$
  monochromatic vertex-disjoint copies of graphs from $\cF$ whose vertex-sets partition $V(G)$.
\end{cor}

\subsection{The Absorption Lemma}\label{subsec:abs}
Given a graph $G$ and $U \subseteq V$, recall that we denote by $G[U]$ the subgraph of $G$ induced
by $U$. Given disjoint sets $V_1, \ldots, V_k \subseteq V(G)$, with $k\geq 2$, we denote by $G[V_1,
\ldots, V_k]$ the subgraph of $G$ with vertex set $V_1 \cup \cdots \cup V_k$ containing only edges
that are between two of the sets $V_1,\ldots,V_k$. Furthermore, for each $v \in V_1$, let
\[
  \deg_G(v,V_2,\ldots,V_k) = |\{ (v_2,\ldots,v_k) \in V_2\times \cdots \times V_k :
  \{v,v_2,\ldots,v_k\} \text{ is a $k$-clique in $G$} \}|
\]
and
\[
  \dd_{G}(v,V_2,\ldots,V_k) := \frac{\deg_{G}(v,V_2,\ldots,V_k)}{|V_2| \cdots |V_k|}.
\]
If additionally, we have an edge colouring $\chi\colon E(G)\to[r]$ of $E(G)$, then we denote by
$\deg_{G,i}(v,V_2,\ldots,V_k) = \deg_{G_i}(v,V_2,\ldots,V_k)$, where $G_i$ is the graph with vertex set $V(G)$ consisting of the edges of $G$ with colour $i$.
We define $\dd_{G,i}(v,V_2,\ldots,V_k)$ similarly and denote $\dd_{G,I}(v,V_2,\ldots,V_k) := \sum_{i\in I} \dd_{G,i}(v,V_2,\ldots,V_k)$, for each $I \subseteq [r]$.
If the graph $G$ is clear from context, we may drop the $G$ in the subscript.

Given a set $V$, we denote by $K(V)$ the complete graph with vertex set $V$. Given disjoint sets $V_1,\ldots,V_k$, we denote by $K(V_1,\ldots,V_k)$ the complete $k$-partite graph with parts $V_1,\ldots,V_k$.
Let $G = K(V_1) \cup K(V_1,\ldots,V_k)$ and let $\cH$ be a collection of subgraphs of $G$. We denote by $\cup\cH$ the graph with edge set $\bigcup_{H \in \cH} E(H)$ and vertex set
$V(\cH) := \bigcup_{H \in \cH} V(H)$.
We say that $\mathcal{H}$ \emph{canonically covers} $V_1$ if
$V_1 \subseteq V(\cH)$ and
\[\card{V(\cH) \cap V_i} \leq \card{V(\cH) \cap V_1}\]
for all $i \in [2,k]$.\footnote{Here, we denote by $[i,j]$ the set of integers $z$ with $i \leq z \leq j$.}
The following lemma is the key ingredient of the proof of our main theorem.

\begin{lemma}[Absorption Lemma]\label{lem:abs}
  There is some absolute constant $K > 0$, such that the following is true for all $d>0$, all integers $\Delta,r \geq 2$ and for every $\mathcal{F} \in \mathcal{F}_{\Delta}$.
  Let $k = \Delta +2$ and let \[C = \exp^2 \left(\left(\frac{r}{d}\right)^{K\Delta} \right).\]

  Consider $k$ disjoint sets $V_1, \ldots, V_k$ with $|V_i| \geq 4|V_1|$, for all $ i \in [2,k]$, and let $G = K(V_1) \cup K(V_1,\ldots,V_k)$.
  Suppose that $\chi: E(G) \to [r]$ is a colouring in which for every $v \in V_1$ we have $\dd_{[r]}(v,V_2,\ldots,V_k) \geq d$.
  Then, there is a collection of at most $C$ vertex-disjoint monochromatic copies of graphs from $\cF$ in $G$ which canonically covers $V_1$.
\end{lemma}

The edges of $G$ inside $V_1$ will only be used to find copies from $\cF$ which lie entirely in $V_1$ in order to greedily cover most vertices of $V_1$. The difficult part is finding monochromatic copies in $K(V_1, \ldots, V_k)$ covering the remaining vertices. To do so, we will reduce the problem to only one colour within $K(V_1, \ldots, V_k)$ and then deduce \cref{lem:abs} from the following lemma.

\begin{lemma}\label{lem:abs-onecolour}
  There is some absolute constant $K > 0$, such that the following is true for all $d>0$, all integers $\Delta,r \geq 2$ and for every $\mathcal{F} \in \mathcal{F}_{\Delta}$.
  Let $k = \Delta +2$ and let
  \[
    C = \exp^2 \left( \left( \frac{r}{d} \right)^{K\Delta} \right).
  \]

  Consider $k$ disjoint sets $V_1, \ldots, V_k$ with $|V_i| \geq 2|V_1|$, for all $ i \in [2,k]$ and let $ G = K(V_1) \cup K(V_1,\ldots,V_k)$.
  Suppose that $\chi: E(G) \to [r]$ is a colouring in which for every $v \in V_1$ we have $\dd_1(v,V_2,\ldots,V_k) \geq d$.
  Then, there is a collection of at most $C$ vertex-disjoint monochromatic copies of graphs from $\cF$ in $G$ which canonically covers $V_1$.
\end{lemma}

\cref{lem:abs} follows routinely from \cref{lem:abs-onecolour}.

\begin{proof}[Proof of \cref{lem:abs}]
Let $K'$ be the absolute constant from \cref{lem:abs-onecolour} and let $d' = d/(2r)$, $\gamma = d'/(kr)$, and $C' = \exp^2 \left( (r/d')^{K'\Delta} \right)$.
Partition $V_1 = U_1 \cup \ldots \cup U_r$ such that for each $j \in [r]$ we have $\dd_j(v,V_2,\ldots,V_k) \geq 2d'$, for all $v \in U_j$. We will inductively cover $U_j$, for each $j \in [k]$.

Let us first consider the base case, i.e., $j=1$. From \cref{prop:greedy}, there is a collection $\cH'$ of at most\footnote{Note that the constant from \cref{prop:greedy} is smaller than $C'$.} $C'$ disjoint monochromatic copies of graphs from $\cF$ covering all but $\gamma |U_1| \leq \gamma |V_1|$ vertices of $G[U_1]$. Let $V'_1 = U_1 \setminus V(\cH')$. By applying \cref{lem:abs-onecolour} to $G' := G[V'_1 \cup V_2 \cup \cdots \cup V_k]$ (with $d'$), there is a collection $\cH''$ of at most $C'$ disjoint monochromatic copies of graphs from $\cF$ in $G'$ which canonically covers $V'_1$. Let $\cH_1 = \cH' \cup \cH''$. Note that $\cH_1$ canonically covers $U_1$ and covers at most $\gamma |V_1|$ vertices of $V_i$, for each $i \in [2,k]$.

Now consider $j\geq 2$ and suppose that we have found a collection $\cH_{j-1}$ of at most $2(j-1)C'$ disjoint monochromatic copies of graphs from $\cF$ in $G$ that canonically covers $U_1 \cup \cdots \cup U_{j-1}$ and covers at most $(j-1)\gamma |V_1|$ vertices of $V_{i}$, for each $i \in [2,k]$.
From \cref{prop:greedy}, there is a collection $\cH'$ of at most $C'$ disjoint monochromatic copies of graphs from $\cF$ covering all but $\gamma |U_j| \leq \gamma |V_1|$ vertices of $G[U_j]$.
Let $V'_1 = U_j \setminus V(\cH')$ and let $V'_i := V_i \setminus V(\cH_{j-1})$, for each $i \in [2,k]$. Note that
\begin{align*}
  |V'_i|
  \geq |V_i| - (j-1)\gamma |V_1|
  \geq 4|V_1| - r\gamma |V_i|
  \geq 2 |V_1|
  \geq 2|V'_1|.
\end{align*}
Also, for each $v \in V'_1$, we have
\begin{align*}
  \deg_j(v,V'_2,\ldots,V'_k)
  & \geq \deg_j(v,V_2,\ldots,V_k) - k (j-1)\gamma |V_2| \cdots |V_k|.
\end{align*}
Consequently,
\begin{align*}
  \dd_j(v,V'_2,\ldots,V'_k)
  & \geq \dd_j(v,V_2,\ldots,V_k) - kr\gamma \geq 2d' - d' \geq d'.
\end{align*}
Therefore, we can apply \cref{lem:abs-onecolour} to $G' := G[V'_1 \cup \cdots \cup V'_k]$ and get a collection $\cH''$ of at most $C'$ disjoint monochromatic copies of graphs from $\cF$ in $G$ that canonically covers $V'_1$. In particular, $\cH''$ covers at most $|V'_1| \leq \gamma |V_1|$ vertices of $V_i$, for each $i\in [2,k]$. Let $\cH_j = \cH_{j-1} \cup \cH' \cup \cH''$. Then $\cH_j$ is a collection of at most $2jC'$ disjoint monochromatic copies of graphs from $\cF$ in $G$ that canonically covers $U_1 \cup \cdots \cup U_j$ and covers at most $j\gamma |V_1|$ vertices of $V_i$, for each $i\in [2,k]$.

In the end, we have found a collection $\cH_r$ of disjoint monochromatic copies of graphs from $\cF$ that canonically covers $V_1$. Furthermore, $\cH_r$ has at most $2rC' \leq \exp^2 \left((r/d)^{4K'\Delta} \right)$ graphs, finishing the proof.
\end{proof}

The proof of \cref{lem:abs-onecolour} is quite long and technical (see \cref{sec:overview} for a sketch), and we will therefore break it up into smaller claims. We use $\claimqed$ to denote the end of the proof of a claim and $\qedsymbol$ to denote the end of the main proof.

\begin{proof}[Proof of \cref{lem:abs-onecolour}.]
Let $\Delta$ and $r$ be given positive integers, $k = \Delta + 2$ and $\cF \in \cF_{\Delta}$.
For each $d>0$, let $C(d)$ be the smallest positive integer $C$ such that the following holds:
\begin{quotation}{$(\star)$}
  Let $V_1,\ldots,V_k$ be disjoint sets with $|V_i| \geq 2|V_1|$ for all $i \in [2,k]$, let $H \subset K(V_1,\ldots,V_k)$ be a graph with $\dd_{H}(v,V_2,\ldots,V_k) \geq d$ for every $v \in V_1$ and $G = K(V_1) \cup H$. Let $\chi: E(G) \to [r]$ be a colouring such that every edge in $E(H)$ receives colour $1$.
  Then, there is a collection $\mathcal{H}$ of at most $C$ vertex-disjoint monochromatic copies of graphs from $\mathcal{F}$ contained in $G$ that canonically covers $V_1$.
\end{quotation}
Note that $C(d)$ is a decreasing function in $d$, and that $C(d) = 0$ for every $d>1$. Our goal is to show that $C(d)$ is finite for every $d > 0$. We will do this by establishing a recursive upper bound (see \cref{ineq:Cdj}).

Let us first define all relevant constants used in the proof. Let $K'$ be the universal constant given by \cref{thm:blowup} and fix some $0 < d \leq 1$. Define
\begin{align*}
  \e = {\left( \frac{d}{100} \right)}^{2K'\Delta},
  \quad \gamma = \tfrac1r \cdot \eps^{k^2 \eps^{-12}}
  \quad \text{ and }
  \quad \eta = \frac{d \gamma^k}{2}.
\end{align*}
It might be of benefit for the reader to have in mind that those constants obey the following hierarchy:
\begin{align*}
  1 \geq d \gg \eps \gg \gamma \gg \eta > 0.
\end{align*}
Furthermore, define
\begin{align*}
  P(d) := 4r^{4rk^2}\log(2/\eta^2) + 1.
\end{align*}
We will prove that for every $d' \geq d$ we have
\begin{align}\label{ineq:Cdj}
  C(d') \leq P(d) + k C\left(d' + \eta \right).
\end{align}
Since $C(d') = 0$ if $d' > 1$, it follows by iterating that $C(d) \leq (2k)^{2/\eta} P(d)$. Furthermore, we have
\begin{align*}
  2/\eta \leq \gamma^{-2k} \leq \e^{-2rk^3\e^{-12}}
  \leq \expb{r\e^{-20}}
  \leq \expb{(r/d)^{400K'\Delta}}.
\end{align*}
It follows that
\[
C(d) \leq \exp^2 \left({(r/d)^{500K'\Delta}}\right) P(d) \leq \exp^2 \left({(r/d)^{1000K'\Delta}}\right)
\]
concluding the proof of \cref{lem:abs-onecolour}.

It remains to prove \cref{ineq:Cdj}. Let $d' \geq d$ be fixed now and let $V_1,\ldots,V_k$, $G$ and $\chi: E(G) \to [r]$ be as in $(\star)$ (with $d'$ playing the role of $d$).
By assumption, there are at least $d|V_1| |V_2| \cdots |V_k|$ cliques of
size $k$ in $G[V_1,V_2,\ldots,V_k]$ each of which is monochromatic in colour $1$.
Since $ \gamma = \eps^{k^2 \eps^{-12}}$ and $d \geq 2k\e$, we can apply \cref{lem:regcyl-partite} to get
some $\gamma' \geq \gamma$ and a $k$-cylinder $Z = (U_1,\ldots,U_k)$ which is $(\e,d/2)$-super-regular with $U_i \subset V_i$ and $|U_i| = \lfloor \gamma' |V_i| \rfloor$ for every $i \in [k]$.
Without loss of generality we may assume that $\gamma |V_i|$ is an integer for every $i \in [k]$ and that we have $\gamma' = \gamma$.
By \cref{prop:greedy}, there is a collection $\cH_R$ of at most $4r^{4rk^2}\log(2/\eta^2)$ vertex-disjoint monochromatic copies of graphs from $\cF$ contained in $K(V_1 \setminus U_1)$ covering all vertices in $V_1 \setminus U_1$ except for a set $R$ with $|R| \leq \eta^2 |V_1|$.
We remark here that
\begin{equation}\label{eq:R-ub}
|R| \leq \eta/(4k) \cdot |U_1| \leq \e^2 |U_1|.
\end{equation}
It remains now to cover the vertices in $R$.
For each $i \in [k]$, let
\begin{align}\label{eq:di}
  d_i = \frac{1 - \gamma^i}{1 - \gamma^k} \cdot d'
\end{align}
and note that $(1-\gamma)d' \leq d_1 \leq \cdots \leq d_k = d'$.
For $i \in [2,k]$, let $\tilde{V}_i = V_i\setminus U_i$ and define
\begin{align*}
  S_{i} &= \{ v \in R : \dd(v,V_2,\ldots,V_{i-1},V_i,U_{i+1},\ldots,U_k) \geq d_i \}, \\
  T_{i} &= \{ v \in R : \dd(v,V_2,\ldots,V_{i-1},\tilde{V}_i,U_{i+1},\ldots,U_k) > d' + 2\eta \}.
\end{align*}
We will prove \cref{ineq:Cdj} using a series of claims, which we shall prove at the end.
\begin{claim}\label{claim:split-R}
  We have $R = S_1 \cup T_2 \cup \ldots \cup T_k$.
\end{claim}
Without loss of generality, we may assume that $S_1, T_2, \ldots, T_k$ are pairwise disjoint (more formally, we can define $T_i':= T_i \setminus (S_1 \cup T_2 \cup \ldots \cup T_{i-1})$ for all $i \in [2,k]$ and continue the proof with these sets).
Our goal now is to cover each of the sets $S_1,T_2,\ldots,T_k$ one by one using the following claims.
\begin{claim}\label{claim:cover-T}
  For every $i \in [2,k]$ and every set $A \subseteq V(G) \setminus T_i$ with $|A \cap V_s| \leq |R|$
  for all $s \in [2,k]$, there is a collection $\cH_i$ of at most $C(d'+\eta)$ monochromatic
  disjoint copies of graphs from $\cF$ in $G$, such that
  \begin{enumerate}[(i)]
    \item $V(\cH_i) \cap V_1 = T_i$,
    \item $V(\cH_i) \cap A = \emptyset$, and
    \item $\card{V(\cH_i) \cap V_j} \leq |T_i|$ for all $j \in [2,k]$.
  \end{enumerate}
\end{claim}
\begin{claim}\label{claim:cover-S}
  For every set $A \subseteq V(G) \setminus (S_1 \cup U_1)$ with $|A \cap V_s| \leq |R|$ for all $s
  \in [2,k]$, there is a monochromatic copy $H_1$ of a graph
  from $\cF$ in $G$, such that
  \begin{enumerate}[(i)]
    \item $V(H_1) \cap V_1 = S_1 \cup U_1$,
    \item $V(H_1) \cap A = \emptyset$ and
    \item $\card{V(H_1) \cap V_j} \leq |S_1 \cup U_1|$ for all $j \in [2,k]$.
  \end{enumerate}
\end{claim}
With these claims at hand, we can now prove \cref{ineq:Cdj}. First, we apply \cref{claim:cover-T}
repeatedly to get collections $\cH_2, \ldots, \cH_k$ of at most $C(d' + \eta)$ disjoint monochromatic copies of graphs from $\cF$ that canonically covers $T_2, \ldots, T_k$, respectively, as follows.
Let $ i \in \{2,\ldots,k\}$ and suppose we have constructed $\cH_2, \ldots, \cH_{i-1}$. Let $A_i :=
V(\cH_2) \cup \ldots \cup V(\cH_{i-1})$ and note that $|A_i \cap V_s| \leq |T_2| + \cdots +
|T_{i-1}| \leq |R|$ for all $s \in [2,k]$. Apply now \cref{claim:cover-T} for $i$ and $A = A_i$ to
get the desired collection $\cH_i$.

Next, we apply \cref{claim:cover-S} with $A = V(\cH_2) \cup \ldots \cup V(\cH_k)$
to get a copy $H_1$ of a graph from $\cF$ with the desired properties. Note that, similarly as above, we have $|A \cap V_s| \leq |R|$ for all $s \in [2,k]$. By construction $V(H_1), V(\cH_2), \ldots, V(\cH_k)$ and
$V(\cH_R)$ are disjoint and cover $V_1$. Moreover, for every $s \in [2,k]$, we have
\begin{align*}
  \card{\left( V(H_1) \cup \ldots \cup V(\cH_k) \cup V(\cH_R) \right) \cap V_s}
  &\leq |S_1 \cup U_1| + |T_1| + |T_2| + \cdots + |T_k|\\
  &\leq |U_1 \cup R| \leq |V_1|.
\end{align*}
Hence, $\{H_1\} \cup \ldots \cup \cH_k \cup \cH_R$ canonically covers $V_1$.
Finally, we have
$\card{\{H_1\} \cup \ldots \cup \cH_k \cup \cH_R} \leq P(d) + k
C\big(d'+\eta\big)$, proving \cref{ineq:Cdj}.
It remains now to prove \cref{claim:split-R,claim:cover-T,claim:cover-S}.
\begin{claimproof}[Proof of \cref{claim:split-R}]
Since  $S_k = R$, it suffices to show $S_i\subseteq S_{i-1} \cup T_i$ for each $i \in [2,k]$.
Let $ i \in [2,k]$ and let $v \in S_i\setminus S_{i-1}$. We have
\begin{align*}
  \deg(v,V_2,\ldots,V_{i-1},\tilde{V}_i,U_{i+1},\ldots,U_k)
  = & \deg(v,V_2,\ldots,V_{i-1},V_i,U_{i+1},\ldots,U_k) \\
    & - \deg(v,V_2,\ldots,V_{i-1},U_i,U_{i+1},\ldots,U_k).
\end{align*}
Therefore,
\begin{align*}
  \dd(v,V_2,\ldots,V_{i-1},\tilde{V}_i,U_{i+1},\ldots,U_k)
  = {} & \dd(v,V_2,\ldots,V_{i-1},V_i,U_{i+1},\ldots,U_k) \frac{|V_i|}{|\tilde{V}_i|}\\
       & - \dd(v,V_2,\ldots,V_{i-1},U_i,U_{i+1},\ldots,U_k) \frac{|U_i|}{|\tilde{V}_i|} \\
  > {} & d_{i}  \frac{|V_i|}{|\tilde{V}_i|} - d_{i-1} \frac{|U_i|}{|\tilde{V}_i|}  \\
  = {} & \frac{d_i - \gamma d_{i-1}}{1-\gamma} \\
  = {} & \frac{(1-\gamma^i)d' - \gamma(1-\gamma^{i-1})d'}{(1-\gamma)(1-\gamma^k)}\\
  = {} & \frac{d'}{1-\gamma^k} \geq d'+2\eta,
\end{align*}
where we use \cref{eq:di} and the definition of $\eta$ to obtain the last identities. Thus $v \in T_i$ and hence $S_i\subseteq S_{i-1} \cup T_i$.
\end{claimproof}
\begin{claimproof}[Proof of \cref{claim:cover-T}]
Let $V'_s := V_s \setminus A$ for all $s \in [2,i-1]$, $\tilde{V}'_i := \tilde{V}_i \setminus A$ and $U'_{s} := U_s \setminus A$ for all $s \in [i+1,k]$. Then, by \cref{eq:R-ub}, we have
\begin{align*}
  |V'_s|
  \geq |V_s| - |R|
  \geq \left( 1 - \tfrac{\eta}{4k} \right) |V_s|
  \geq \frac{|V_s|}{2},
  & \text{ for } s=2,\ldots,i-1,\\
  |\tilde{V}'_i|
  \geq |\tilde{V}_i| - |R|
  \geq \left( 1 - \tfrac{\eta}{4k} \right) |\tilde{V}_i|
  \geq \frac{|\tilde{V}_i|}{2},
  & \text{ and } \\
  |U'_s|
  \geq |U_s| - |R|
  \geq \left( 1 - \tfrac{\eta}{4k} \right) |U_s|
  \geq \frac{|U_j|}{2},
  & \text{ for } s=i+1,\ldots,k.
\end{align*}
In particular, it follows that
\begin{align*}
  |V_s \setminus V'_s| \leq |R| \leq \frac{\eta}{4k} |V_s|
  \leq \frac{\eta}{2k} |V'_s|, &\text{ for } s=2,\ldots,i-1,\\
  |V_i \setminus V'_i| \leq |R| \leq \frac{\eta}{4k} |V_i|
  \leq \frac{\eta}{2k} |V'_i|, &\text{ and } \\
  |U_s \setminus U'_s| \leq |R| \leq \frac{\eta}{4k} |U_s|
  \leq \frac{\eta}{2k} |U'_s|, &\text{ for } s=i+1,\ldots,k.
\end{align*}
Therefore, for every $v \in T_i$, we have
\begin{align*}
  &\dd(v,V'_2,\ldots,V'_{i-1},\tilde{V}'_i,U'_{i+1},\ldots,U'_k)\\
  &\geq d'+2\eta
      - \sum_{s=2}^{i-1} \frac{|V_s\setminus V'_s|}{|V'_s|}
      -\frac{|\tilde{V}_i\setminus \tilde{V}'_i|}{|\tilde{V}'_i|}
      - \sum_{s=i+1}^k \frac{|U_{s}\setminus U'_{s}|}{|U'_{s}|}\\
  &\geq d'+2\eta - (k-1) \frac{\eta}{2k} \geq d' + \eta.
\end{align*}
Hence, by definition of $C(d' + \eta)$ (see $(\star)$), there exists a collection $\mathcal{H}_{i}$ of at most $C(d' + \eta)$ monochromatic copies of graphs from $\cF$ that canonically covers $T_{i}$ in the graph
\begin{align*}
  K(T_i) \cup K(T_i,V'_2,\ldots,V'_{i-1},\tilde{V}'_i,U'_{i+1},\ldots,U'_k).
\end{align*}
By construction, $\cH_i$ satisfies the requirements of the claim (note that $(iii)$ holds since $\cH_i$ is a canonical covering).
\end{claimproof}
\begin{claimproof}[Proof of \cref{claim:cover-S}]
  Let $Y_1 = S_1$ and, for each $i \in [2,k]$, let $X_i = U_i \cap A$. Observe that $|Y_1| \leq |R| \leq \e^2 |U_1|$ and $|X_i| \leq |R| \leq \e^2 |U_i|$ for all $i \in [2,k]$.
  Let $U'_1 = U_1 \cup Y_1$ and, for each $i \in [2,k]$, let $U'_i := U_i \setminus X_i$.
  We now consider the cylinder $Z' := (U'_1,\ldots,U'_k)$. By definition of $S_1$, we have
  $\dd(v,U_2,\ldots,U_k) \geq d_1 \geq d/2$ and in particular $\deg(v,U_i) \geq d/2 \cdot
  |U_i|$ for all $ v \in Y_1$ and $i \in [2,k]$.

  Hence, by \cref{lem:robust}, $Z'$ is $(8\eps,d/4)$-super-regular. Furthermore, we have $|U_1'| \leq |U_i'|$ for all $i \in [k]$. Thus, by \cref{thm:blowup}, there is a monochromatic copy $H_1$ of a graph from $\cF$ in $Z$ that covers $U_1' = U_1 \cup S_1$ and at most $|U_1'|$ vertices from each of $U_2', \ldots, U_k'$. By construction, this copy satisfies the requirements of the claim.
\end{claimproof}
This finishes the proof of \cref{lem:abs-onecolour}.
\end{proof}

\subsection{Proof of \texorpdfstring{\cref{thm:main}}{the main theorem}}\label{subsec:main}
In this section, we will finish the proof of \cref{thm:main}. We will make use of the following
lemma from~\cite{Bustamante2019} and follow the same proof technique. Since our proof of this lemma
is short, we include it here for completeness. Given a $k$-uniform hypergraph $\mathcal{H}$, a
vertex $v \in V(\mathcal{H})$ and sets $B_2, \ldots, B_k \subseteq V(\mathcal{H})$, we define
\[
  \deg_{\mathcal{H}}(v,B_2,\ldots,B_k) := \left|\left\{ (v_2, \ldots, v_k) \in B_2 \times \ldots
    \times B_k: \{v,v_2,\ldots,v_k\} \in E(\mathcal{H}) \right\}\right|.
\]
\begin{lemma}\label{lem:indtr}
  Let $k$ and $N$ be positive integers and let $\mathcal{H}$ be a $k$-uniform
  hypergraph. Suppose that $B_1, \ldots,B_N \subseteq V(\mathcal{H})$ are
  non-empty disjoint sets such that for every $1\leq i_1 < \cdots < i_k
  \leq N$ we have
  \[
    \deg_{\mathcal{H}}(v,B_{i_2}, \ldots, B_{i_k}) < {\binom{N}{k}}^{-1}
    |B_{i_2}| \cdots |B_{i_k}|
  \]
  for all $ v \in B_{i_1}$. Then, there exists an independent set $\left\{ v_1,
    \ldots, v_N \right\}$ with $v_i \in B_i$, for each $i \in [N]$.
\end{lemma}
\begin{proof}
  For each $i \in [N]$, let $v_i$ be chosen uniformly at random from $B_i$. Let
  $I = \{v_1,\ldots,v_N\}$. Then
  we have
  \begin{align*}
    \Pro{\text{$I$ is not an independent set}}
    & \leq \sum_{1\leq i_1 < \cdots < i_k \leq N} \Pro{\{v_{i_1},\ldots,v_{i_k}\} \in E(\mathcal{H})} \\
    & = \sum_{1\leq i_1 < \cdots < i_k \leq N} \frac{1}{|B_{i_1}|} \sum_{v \in B_1} \Pro{\{v_{i_1},\ldots,v_{i_k}\} \in E(\mathcal{H}) \;|\; v_{i_1} = v} \\
    & = \sum_{1\leq i_1 < \cdots < i_k \leq N} \frac{1}{|B_{i_1}|} \sum_{v \in B_1} \frac{\deg_{\mathcal{H}}(v,B_{i_2}, \ldots, B_{i_k})}{|B_{i_2}| \cdots |B_{i_k}|} \\
    & < \sum_{1\leq i_1 < \cdots < i_k \leq N} \binom{N}{k}^{-1} = 1.
  \end{align*}
  Therefore, there exists an independent set $\{v_1,\ldots,v_N\}$ with $v_i\in B_i$, for each
  $i\in[N]$.
\end{proof}
We are now able to prove \cref{thm:main}. The main idea is to find reasonably large cylinders that are super-regular for one of the colours, greedily cover most of the remaining vertices using \cref{prop:greedy} and then apply the Absorption Lemma (\cref{lem:abs}) to the set of remaining vertices that share many monochromatic cliques with the cylinders.
We then iterate this process until no vertices remain. Using \cref{lem:indtr}, we will show that a bounded number of iterations suffices.
\begin{proof}[Proof of \cref{thm:main}]
Fix $r, \Delta \geq 2$, $\cF \in {\cF}_{\Delta}$. Let $G$ be an $r$-edge-coloured complete graph on $n$
vertices. Let
\[
  k= \Delta +2,
  \quad N = r^{rk},
  \quad \delta = \binom{N}{k}^{-1}
  \quad \text{ and }
  \quad d = \frac{1}{2r}.
\]
In order to use \cref{thm:blowup}
and \cref{lemma:regcyl}, respectively, consider the constants
\[
  \eps = {(\delta d^{\Delta})}^{2K'}
  \quad \text{ and }
  \quad \gamma = \e^{r^{8rk}\e^{-5}},
\]
where $K'$ is the universal constant given by \cref{thm:blowup}. Consider also the constants
\[
  \alpha = \eps^2
  \quad \text{ and }
  \quad C_1 = 4r^{2rk\Delta}\log\left(\frac{4}{\alpha\gamma}\right)
\]
in order to use \cref{prop:greedy}. Finally, let
\[
C_2 = \exp^{2}({(2r/\delta)}^{\tilde{K}\Delta}) \leq \exp^{2}\left( r^{16\tilde{K} r \Delta^3} \right),
\]
where $\tilde{K}$ is the universal constant from \cref{lem:abs}, and
let $K = 20\tilde{K}$.

We will build a framework consisting of many $k$-cylinders working as absorbers and small sets that
can be absorbed by them. More precisely, our goal is to define sets with the following properties
(\cref{fig:proof} should help the reader to understand the structure of those sets as we define
them):

\begin{quotation}
  \noindent
  \textbf{Framework.} There are sets $Z_1, \ldots, Z_N$, $S_{k-1}, \ldots, S_N$,
  $R_k, \ldots, R_{N+1}$, $R'_k, \ldots, R'_{N+1}$ with the following properties.
  \begin{enumerate}[(F.1)]
  \item\label{item:F-partition} $V(G) = \bigcup_{i=1}^{N} Z_i \cup \bigcup_{i=k-1}^N S_i \cup
    \bigcup_{i=k}^{N+1} R'_i$ is a partition.
  \item $Z_1, \ldots, Z_N$\footnote{We shall identify the cylinders with their vertex-set.} are
    $k$-cylinders which are $(\e,d)$-super-regular in one of the colours (or empty).
  \item $S_{k-1}, \ldots, S_N$ are sets of vertices which we will cover greedily by monochromatic
    copies of graphs from $\cF$.
  \item For each $i \in [k,N+1]$, $R'_i$ can be partitioned into sets $R_{i,I}'$ for all $I \in
    \binom{[i-1]}{k-1}$, such that, for each $I = \{ i_1, \ldots, i_{k-1}\} \subseteq [i]$, we have
    $\dd_{[r]}(u,Z_{i_1}, \ldots, Z_{i_{k-1}}) \geq \delta$ for all $u \in R'_{i+1,I}$.
  \item For each $k \leq i < j \leq N+1$, we have $S_j \cup Z_j \cup R'_j \subseteq R_i$ and $|R_i| \leq
    \alpha |Z_{i-1}|$.
  \end{enumerate}
\end{quotation}

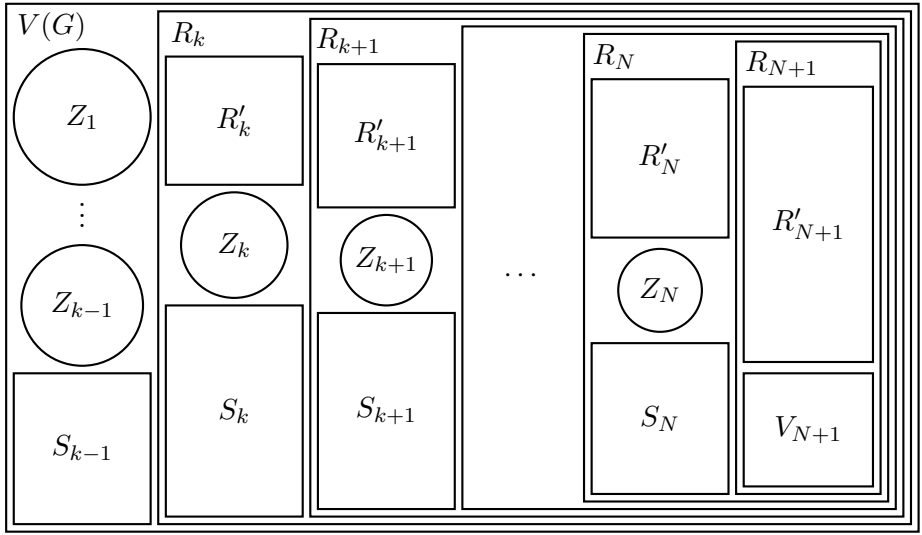
\begin{figure}
  \centering
  \begin{tikzpicture} [scale=1, thick, auto,
    cylinder/.style={circle, draw, inner sep=0pt, outer sep=0pt},
    leftover/.style={rectangle, draw, anchor=center, inner sep=0pt, outer sep=0pt},
    centeredlabel/.style={midway, anchor=center},
    ]

    \coordinate (a) at (0,0);
    \coordinate (b) at (12,7);

    \draw (a) rectangle (b);
    \node at ($(.6,6.7)$) {$V(G)$};

    \node [cylinder, minimum width=1.8cm](u1) at ($(a)+(1,5.5)$) {$Z_1$};

    \node at ($(a)+(1.0,4.3)$) {$\vdots$};

    \node [cylinder, minimum width=1.6cm](u2) at ($(a)+(1,3)$) {$Z_{k-1}$};

    \node [leftover,minimum width=1.8cm, minimum height=2cm](s3) at ($(a)+(1,1.1)$) {$S_{k-1}$};

    \draw ($(a)+(2,.1)$) rectangle ($(b)+(-.1,-.1)$);
    \node at ($(2.4,6.6)$) {$R_k$};

    \draw ($(a)+(2.1,4.6)$) rectangle ($(b)+(-8.1,-.7)$) node [centeredlabel]{$R'_k$};

    \node [cylinder, minimum size=1.4cm](u3) at ($(a)+(3,3.8)$) {$Z_k$};

    \node [leftover,minimum width=1.8cm, minimum height=2.8cm](s3) at ($(a)+(3,1.6)$) {$S_k$};

    \draw ($(a)+(4,.2)$) rectangle ($(b)+(-.2,-.2)$);
    \node at ($(4.5,6.5)$) {$R_{k+1}$};

    \draw ($(a)+(4.1,4.3)$) rectangle ($(b)+(-6.1,-.8)$) node [centeredlabel]{$R'_{k+1}$};

    \node [cylinder, minimum width=1.2cm](u4) at ($(a)+(5,3.6)$) {$Z_{k+1}$};

    \node [leftover,minimum width=1.8cm, minimum height=2.6cm](s3) at ($(a)+(5,1.6)$) {$S_{k+1}$};

    \draw ($(a)+(6,.3)$) rectangle ($(b)+(-.3,-.3)$);

    \node (u5) at ($(a)+(6.8,3.4)$) {$\cdots$};

    \draw ($(a)+(7.6,.4)$) rectangle ($(b)+(-.4,-.4)$);
    \node at ($(8.0,6.3)$) {$R_N$};

    \draw ($(a)+(7.7,3.9)$) rectangle ($(b)+(-2.5,-1.0)$) node [centeredlabel]{$R'_N$};

    \node [cylinder, minimum width=1.1cm](u4) at ($(a)+(8.6,3.2)$) {$Z_N$};

    \node [leftover,minimum width=1.8cm, minimum height=2.0cm](s3) at ($(a)+(8.6,1.5)$) {$S_N$};

    \draw ($(a)+(9.6,.5)$) rectangle ($(b)+(-.5,-.5)$);
    \node at ($(10.2,6.2)$) {$R_{N+1}$};

    \draw ($(a)+(9.7,2.25)$) rectangle ($(b)+(-0.6,-1.1)$) node [centeredlabel]{$R'_{N+1}$};

    \node [leftover,minimum width=1.7cm, minimum height=1.5cm](s3) at ($(a)+(10.55,1.35)$) {$V_{N+1}$};

  \end{tikzpicture}
  \caption{A partition of $V(G)$. Each set in the picture is much smaller than the closest cylinder $Z_i$ to the left.}\label{fig:proof}
\end{figure}

So let us construct those sets from the framework. First, if $n < 1/4\gamma$, then \cref{cor:greedy}
gives a covering with at most $C_2$ monochromatic vertex-disjoint copies of graphs from $\cF$.
Therefore we may assume that $n \geq 1/4\gamma$. Hence, by applying \cref{lemma:regcyl} multiple
times, we find $k-1$ vertex-disjoint $k$-cylinders $Z_1, \ldots, Z_{k-1}$ such that each of them is
$(\e,d)$-super-regular in some colour (not necessarily the same) and $ \card{Z_1} \geq \cdots \geq
\card{Z_{k-1}} \geq \gamma n/2$. Let $V_{k-1} = V(G) \setminus (Z_1 \cup \cdots \cup Z_{k-1})$. By
\cref{prop:greedy}, there is a collection of at most $C_1$ monochromatic vertex-disjoint copies from
$\cF$ in $V_{k-1}$ covering a set $S_{k-1}$ such that the leftover vertices $R_k = V_{k-1} \setminus
S_{k-1}$ satisfies $|R_k| \leq \alpha \gamma n/2 \leq \alpha |Z_{k-1}|$. Let $R_k' \subseteq
R_k$ be the set of vertices $ u \in R_k$ with $ \dd_{[r]}(u,Z_1, \ldots, Z_{k-1}) \geq \delta$. Let
$R'_{k,[k-1]} = R'_k$ and $V_k = R_k \setminus R_k'$.

Inductively, for each $ i = k, \ldots, N $, we do the following.
If $|V_i| < 1/4\gamma$, we use \cref{cor:greedy} to cover $V_i$ using at most $C_2$ monochromatic vertex-disjoint copies from $\cF$ and let $Z_i=S_{i}=R_{i+1}=R'_{i+1}=V_{i+1}=\emptyset$.
Otherwise, we apply \cref{lemma:regcyl} to find a monochromatic $(\e,d)$-super-regular $k$-cylinder $Z_i$ contained in $V_i$ with $\card{Z_i} \geq \gamma \card{V_i}$.
By \cref{prop:greedy}, there is a collection of at most $C_1$ monochromatic, vertex-disjoint copies from $\cF$ in $ V_i \setminus Z_i $ covering a set $S_i \subseteq V_i$, so that the set of leftover vertices $R_{i+1} = V_i \setminus S_i$ has size at most $\alpha \gamma|V_i| \leq \alpha \card{Z_i}$.

Let $R'_{i+1}$ be the set of vertices $u$ in $R_{i+1}$ for which there is a set $I = \{i_1,\ldots,i_{k-1}\} \subseteq [i]$ such that $\dd_{[r]}(u,Z_{i_1},\ldots,Z_{i_{k-1}}) \geq \delta$. Let
\[
R_{i+1}' = \bigcup_{I \in \binom{[i]}{k-1}} R'_{i+1,I}
\]
be a partition of $R'_{i+1}$ so that, for each $I = \{ i_1, \ldots, i_{k-1}\} \subseteq [i]$, we have $\dd_{[r]}(u,Z_{i_1}, \ldots, Z_{i_{k-1}}) \geq \delta$ for all $u \in R'_{i+1,I}$.
Finally, let $V_{i+1} = R_{i+1} \setminus R_{i+1}'$.

The following claim implies that these sets partition $V(G)$ as in \cref{item:F-partition}.
\begin{claim}\label{claim:2}
  The set $V_{N+1}$ is empty.
\end{claim}
\begin{claimproof}
  Define a $k$-uniform hypergraph $\mathcal{H}$ with vertex set $U = Z_1 \cup \ldots \cup Z_{N} \cup
  V_{N+1}$ and hyperedges corresponding to monochromatic $k$-cliques in $G[U]$. If $V_{N+1}$ is
  non-empty, then so are $Z_1, \ldots, Z_N$. Since for each $i=k,\ldots,N$ we have $Z_i \subseteq
  R_i\setminus R'_i$ and $V_{N+1} = R_{N+1}\setminus R'_{N+1}$, it follows that $\mathcal{H}$ satisfies the hypothesis of \cref{lem:indtr}.
  Therefore, there is an independent set $\{v_1,\ldots,v_{N+1}\}$ in $\mathcal{H}$ of size $N+1$. On the other hand, since $N \geq R_r(K_k)$, it follows that the set $\{v_1,\ldots,v_{N+1}\}$ has a
  monochromatic $k$-clique in $G[U]$, which is a contradiction.
\end{claimproof}
The vertices in $S_{k-1} \cup \cdots \cup S_N$ are already covered by monochromatic copies of graphs
from $\mathcal{F}$. Our goal now is to cover the sets $R'_k,\ldots,R'_{N+1}$ using \cref{lem:abs} without using too many vertices from the cylinders $Z_1,\ldots,Z_N$.
This way, we can cover the remaining vertices in $Z_1\cup\cdots\cup Z_N$ using \cref{thm:blowup}.
\begin{claim}\label{claim:1}
  Let $i \in \{k, \ldots, N+1\}$ and $I = \{i_2,\ldots,i_k\} \subseteq [i-1]$. Let
  $A \subseteq V(G) \setminus R_{i,I}$
  be a set with $\card{A \cap Z_{j}} \leq \alpha \card{Z_{j}}$ for each $j \in I$.
  Then there is a collection of at most $C_2$ monochromatic vertex-disjoint copies of
  graphs from $\cF$ in
  \[
    G' = K(R'_{i,I}) \cup K(R'_{i,I},Z_{i_2},\ldots,Z_{i_k})
  \]
  which are disjoint from $A$ and canonically cover $R'_{i,I}$.
\end{claim}
\begin{claimproof}
Let $\tilde V_1 = R_{i,I}'$ and for $j \in [k]\setminus\{1\}$, let $\tilde V_j = Z_{i_j} \setminus A$. Note that $|\tilde V_j| \geq 4|\tilde V_1|$ for every $j \in [k]\setminus\{1\}$ and
\begin{align*}
  \deg_{[r]}(v, \tilde V_2, \ldots, \tilde V_k)
  \geq & \deg_{[r]}(v,Z_{i_2}, \ldots, Z_{i_k}) - k\alpha |Z_{i_2}| \cdots |Z_{i_k}| \\
  \geq & (\delta - k \alpha) |Z_{i_2}| \cdots |Z_{i_k}| \\
  \geq & \delta/2 \cdot |Z_{i_2}| \cdots |Z_{i_k}|
\end{align*}
for every $v \in \tilde V_1$. Hence, by \cref{lem:abs}, there is a collection of at most $C_2$
vertex-disjoint copies from $\cF$ in $\tilde V_1 \cup \ldots \cup \tilde V_k$ that canonically
covers $ \tilde V_1$, finishing the proof.
\end{claimproof}
We will use \cref{claim:1} now to cover $\bigcup_{i=k}^{N+1} R'_{i}$.
Let $\prec$ be a linear order on $\cI := \left\{(i,I): i \in [k,N+1], I \in \binom{[i-1]}{k-1}\right\}$.
Let $(i,I) \in \cI$ and suppose that, for all $(i',I') \in \cI$ with $(i',I') \prec (i,I)$, we have already constructed a family $\cH_{i'.I'}$ of monochromatic copies of graphs from $\cF$ which canonically covers $R'_{i',I'}$ in $K(R'_{i',I'}) \cup K(R'_{i',I'},Z_{i_2'},\ldots,Z_{i_k'})$, where $I' = \{i'_2, \ldots, i'_k\}$, and such that the sets $V(\cH_{i',I'})$, for $(i',I')\prec (i,I)$, are disjoint.

Let $A = \bigcup_{(i',I') \prec (i,I)} V(\cH_{i',I'})$ be the set of already covered vertices.
We claim that
\begin{equation}\label{eq:AcapZ-ub}
  \card{A \cap Z_j} \leq \alpha |Z_j|
\end{equation}
for each $j \in [N]$.
Indeed, given some $j \in [N]$, for all $(i',I') \in \cI$ with $i' \leq j$, we have $V(\cH_{i',I'}) \cap Z_j = \emptyset$, since $\cH_{i',I'}$ canonically covers $R'_{i',I'}$ in $K(R'_{i',I'}) \cup K(R'_{i',I'},Z_{i_2'},\ldots,Z_{i_k'})$.
Now for all $(i',I') \in \cI$ with $i' > j$, we have $\card{V(\cH_{i',I'}) \cap Z_j} \leq |R_{i',I'}'|$, again because $\cH_{i,I}$ canonically covers $R'_{i',I'}$.
Therefore,
\begin{align*}
  \card{A \cap Z_j}
  \leq \sum_{(i',I') \prec (i,I)} \card{V(\cH_{i',I'}) \cap Z_j}
  \leq \sum_{(i',I') \in \cI \; : \; i' > j} \card{R'_{i',I'}}
  \leq \card{R_{j+1}},
\end{align*}
since the sets $\{R_{i',I'}': (i',I') \in \cI, i > j\}$ are disjoint subsets of $R_{j+1}$.
Finally, since $\card{R_{j+1}} \leq \alpha \card{Z_j}$, this implies \cref{eq:AcapZ-ub}.
In particular, by \cref{claim:1}, there is a collection $\cH_{i,I}$ of monochromatic copies of graphs from $\cF$ that canonically covers $R'_{i,I}$ in $K(R'_{i,I}) \cup K(R'_{i,I},Z_{i_2},\ldots,Z_{i_k})$, where $I = \{i_2, \ldots, i_k\}$, and such that $V(\cH_{i,I})$ is disjoint from $A$.

It remains to cover $\bigcup_{i=1}^N Z_i$. Let $A := \bigcup_{(i,I) \in \cI} V(\cH_{i,I})$ be the
set of vertices covered in the previous step. Note that, similarly as in \cref{eq:AcapZ-ub}, we have
$|A \cap Z_j| \leq \alpha |Z_j|$ for all $j \in [N]$. Therefore, by \cref{lem:robust},
the cylinder $\tilde Z_j$ obtained from $Z_j$ by removing all vertices in $A$ is
$(8\e,d/2)$-super-regular and $\e$-balanced for every $j \in [N]$. It follows from \cref{thm:blowup}
that, for every $j \in [N]$, there is a collection $\cH_j$ of at most $\Delta+3$ monochromatic
vertex-disjoint copies of graphs from $\cF$ contained in $Z_j$ covering $V(Z_j)$.

In total, the number of monochromatic copies we used to cover $V(G)$ is at most
 \begin{align*}
   N \cdot C_1 + N^k \cdot C_2 + N \cdot (\Delta+3)
   &\leq 2N^k C_2\\
   &\leq 2{r}^{rk^2} \cdot \exp^2\left(r^{16 \tilde K r \Delta^3} \right)\\
   &\leq \exp^2\left(r^{K r \Delta^3} \right).
 \end{align*}
 This concludes the proof of \cref{thm:main}.
\end{proof}

\section{Concluding Remarks}\label{sec:concluding-remarks}
We were able to prove that sequences of graphs with maximum degree $\Delta$ have finite $r$-colour tiling number for every $r \geq 3$, but our bound is super-exponential in $\Delta$.
Grinshpun and S\'ark\"ozy~\cite{Grinshpun2016} conjectured that it is possible to prove an upper bound which is essentially exponential in $\Delta$ (see \cref{conj:main}). The problem becomes somewhat easier when restricted to bipartite graphs. In fact, our proof gives a double exponential upper bound in $\Delta$ for $r$-colour tiling numbers of sequences of bipartite graph with maximum degree $\Delta$. Indeed, the factor $k$ in the recursive bound \cref{ineq:Cdj} can be dropped for bipartite graphs. It would be very interesting to confirm \cref{conj:main} for sequences of bipartite graphs.

Another interesting problem is to prove a version of \cref{thm:main} for other sequences of graphs. Given a sequence of graphs $\cF = \{F_i: i\in\mathbb{N}\}$ with $|F_i| = i$, for every $i\in
\mathbb{N}$, let $\rho_r(\cF) = \sup_{i \in \mathbb{N}} R_r(F_i)/i$. If $\rho_r(\cF)$ is finite, then we say
that $\mathcal{F}$ has linear $r$-colour Ramsey number. If $\cF$ is \emph{increasing}\footnote{That
  is, $F_i \subseteq F_{i+1}$, for every $i \in \bN$.}, then it follows from the pigeon-hole
principle that $\tau_r(\mathcal{F}) \geq \rho_r(\mathcal{F})$. Indeed, for each $n\in \mathbb{N}$,
every $r$-edge-coloured $K_n$ contains a monochromatic copy from $\cF$ of size at least $i = \lceil
n/\tau_r(\cF) \rceil$. In particular, since $\cF$ is increasing, there is a monochromatic copy of
$F_i$ in every $r$-edge colouring of $K_n$. This implies that $R_r(F_i) \leq \tau_r(\cF) \cdot i$, and
therefore $\rho_r(\mathcal{F}) \leq \tau_r(\mathcal{F})$.

Graham, R\"odl and Ruci\'nski~\cite{Graham2000} proved that there exists a sequence of bipartite graphs
$\cF= \{F_i:i\in\mathbb{N}\}$ with $\rho_2(\cF) \geq 2^{\Omega(\Delta)}$. Grinshpun and S\'ark\"ozy
observed that one can make this sequence increasing, thereby showing that $\tau_2(\cF) \geq
2^{\Omega(\Delta)}$ as well. Conlon, Fox and Sudakov~\cite{Conlon2012a} proved that for every sequence of
graphs with degree at most $\Delta$, we have $\rho_2(\mathcal{F}) \leq 2^{O(\Delta \log{\Delta})}$
while Grinshpun and S\'ark\"ozy~\cite{Grinshpun2016} proved that $\tau_2(\cF) \leq 2^{O(\Delta
  \log{\Delta})}$. For more colours, Fox and Sudakov~\cite{Fox2009} proved that for every sequence of
graphs with degree at most $\Delta$, we have $\rho_r(\mathcal{F}) \leq 2^{O_r(\Delta^2)}$, while our
main result shows that $\tau_r(\mathcal{F}) \leq \exp^3(O_r(\Delta^3))$.

With these results in mind, one can naturally ask if there exists a function
$f:\mathbb{R}\to\mathbb{R}$ such that for every sequence of graphs $\cF =
\{F_i:i\in\mathbb{N}\}$ we have $\tau_r(\cF) \leq f(\rho_r(\cF))$. That is, if it is possible to bound $\tau_r(\cF)$ in terms of $\rho_r(\cF)$. In particular, this would imply that sequences of graphs with linear Ramsey number have finite tiling number. However, the following example due to Alexey Pokrovskiy (personal communication) shows that $\tau_r(\cF)$ cannot be bounded
by $\rho_r(\cF)$ in general.
Let $S_i$ be a star with $i$ vertices and let $\mathcal{S} =
\{S_i:i\in\mathbb{N}\}$ be the family of stars. It follows readily from the pigeonhole principle that $R_r(S_i) \leq r(i-2) + 2$, for every $i \in \bN$, and thus $\rho_r(\mathcal{S}) \leq r$.
However, the following shows that $\tau_r(\mathcal{S}) = \infty$, for every $r \geq 2$.

\begin{example}\label{ex:stars}
  For all $r\geq 2$ and all sufficiently large $n$, we have $\tau_r(\mathcal{S},n) \geq r \cdot
  \log(n/8)$.
\end{example}

\begin{proof}
Let $ \tau = r\log(n/8)$ and colour $E(K_n)$ uniformly at random with $r$ colours. Given a
vertex $v \in [n]$ and a colour $c$, let $S_c(v)$ be the star centred at $v$ formed by all the
edges of colour $c$ incident on $v$. Note that there is a monochromatic $\mathcal{S}$-tiling of size at most $\tau$ if and only if there are distinct vertices $v_1, \ldots, v_\tau$ and colours $c_1, \ldots, c_\tau \in [r]$ such that $\bigcup_{i \in [\tau]} V(S_{c_i}(v_i)) = [n]$.

Fix distinct vertices $v_1, \ldots, v_\tau \in [n]$ and colours $c_1, \ldots, c_\tau \in [r]$. Let
$U$ be the random set $U = \bigcup\nolimits_{i \in [\tau]} V(S_{c_i}(v_i))$. Note that the events
$\{v \in U\}$, for $v\in[n]\setminus\{v_1,\ldots,v_{\tau}\}$, are independent and each has
probability $1 - {(1-1/r)}^{\tau}$. Therefore, using $ e^{-x/(1-x)} \leq 1 - x \leq e^x$ for all $x
\leq 1$, we get
\begin{align*}
  \Pro{U = [n]}
  &= {\left( 1- {\left(1-1/r\right)}^{\tau} \right)}^{n-\tau}\\
  &\leq \exp\left( -(n-\tau) {\left(1-1/r\right)}^{\tau} \right)\\
  &\leq \exp\left(-n{(1-1/r)}^{\tau+1}\right)\\
  &\leq \exp\left(-n\exp\left(-4\tau/r\right)\right)\\
  &\leq \exp\left(-\sqrt{n}\right).
\end{align*}

Taking a union bound over all choices of $v_1, \ldots, v_\tau$ and $c_1, \ldots, c_\tau$, we conclude that the probability that there is a monochromatic $\mathcal{S}$-tiling of size $\tau$ is at most
\[{(rn)}^{-\tau} \cdot e^{-\sqrt{n}} < 1\]
for all sufficiently large $n$.
Hence, there exists an $r$-colouring of $E(K_n)$ without a monochromatic $\mathcal{S}$-tiling of size at most $\tau$, finishing the proof.
\end{proof}

Lee~\cite{Lee2017} proved that graphs with bounded degeneracy\footnote{A graph $G$ is $d$-degenerate if there is an ordering of its vertices so that every $v \in V(G)$ is adjacent to at most $d$ vertices which come before $v$.} have linear Ramsey number.
\cref{ex:stars} shows however that it is not possible to extend this result to a tiling result. Nevertheless, it may be possible to allow unbounded degrees in this case.

\begin{question}\label{conj:bounded-degeneracy}
  Is there a function $\omega: \bN \to \infty$ with $\lim_{n \to \infty} \omega(n) = \infty$, such that the following is true for all integers $r,d \geq 2$?
  If $\cF= \{F_1, F_2, \ldots\}$ is a sequence of $d$-degenerate graphs with $v(F_n) = n$ and $\Delta (F_n) \leq \omega(n)$ for all $n \in \mathbb{N}$, then $\tau_r(\cF) < \infty$.
\end{question}

B\"ottcher, Kohayakawa and Taraz~\cite{Boettcher2015} proved an extension of the blow-up lemma to graphs $H$ of bounded arrangeability\footnote{A graph $G$ is called \emph{$a$-arrangeable} for some $a \in \bN$ if its vertices can be ordered in such a way that for every $v \in V(G)$, there are at most $a$ vertices to the left of $v$ that have some common neighbour with $v$ to the right of $v$.}
with $\Delta(H) \leq \sqrt{n}/\log(n)$.
Using their result, it is possible to prove the following strengthening of \cref{thm:main}.
\begin{thm}\label{thm:arrange}
	For all integers $r,a \geq 2$ and all sequences of $a$-arrangeable graphs $\cF = \{F_1, F_2,
  \ldots \}$ with $|F_n| = n$ and $\Delta (F_n) \leq \sqrt n/\log(n)$ for all $n \in \mathbb{N}$, we
  have $ \tau_r(\cF) < \infty$.
\end{thm}
The proof is almost identical, with the following two differences. First, instead of
\cref{thm:blowup}, we need to use the blow-up lemma mentioned above together with the following
alternative to Hajnal's and Szemer\'edi's theorem which guarantees balanced partitions of graphs
with small degree. Given a sequence $\cF = \{F_1, F_2, \ldots\}$ of $a$-arrangeable graphs with
$\Delta(F_n) \leq \sqrt n/\log(n)$ for every $n \in \mathbb{N}$, we define another sequence of
graphs $\tilde \cF = \{\tilde F_1,\tilde F_2, \dots \}$ as follows. Since every $a$-arrangeable
graph is $(a+2)$-colourable, we can fix a partition of $V(F_n) = V_1(F_n) \cup \ldots \cup
V_{k}(F_n)$ into independent sets, where $k = a+2$. Then, for every $j \in \bN$, we define
$\tilde{F_{jk}}$ to be the disjoint union of $k$ copies of $F_j$. Note that each $\tilde F_{jk}$ has
a $k$-partition into parts of equal sizes (by rotating each copy around). Finally, for each $j \in
\bN \cup \{0\}$ and every $i \in [k-1]$, we define $\tilde F_{jk+i}$ to be the disjoint union of
$\tilde F_{jk}$ and $i$ isolated vertices (here $\tilde F_0$ is the empty graph). Observe that all
$\tilde{F_n}$ have $k$-partitions into parts of almost equal sizes. Furthermore, every
$\tilde{\cF}$-tiling $\cT$ corresponds to an $\cF$-tiling $\tilde \cT$ of size at most $(2k-1)
|\cT|$. Therefore, it suffices to prove \cref{thm:arrange} for graphs with balanced
$(a+2)$-partitions.

Second, we need to replace \cref{thm:bounded-ramsey} with a similar theorem for $a$-arrangeable
graphs $G$ with $\Delta(G) \leq \sqrt n/\log(n)$, where $n = v(G)$. For two colours, such a theorem
was proved by Chen and Schelp~\cite{Chen1993}. For more than two colours, this was (to the best of the
author's knowledge) never explicitly stated, but is easy to obtain using modern tools (for example,
by applying the above mentioned blow-up lemma for $a$-arrangeable graphs).

\section{Acknowledgement}
\label{sec:acknowledgement}

The authors would like to tank the organisers of the workshop Extremal and Structural Combinatorics, held at IMPA in Rio de Janeiro, where this work began, and to thank Rob Morris for reading a previous version of this paper.

\providecommand{\bysame}{\leavevmode\hbox to3em{\hrulefill}\thinspace}
\providecommand{\MR}{\relax\ifhmode\unskip\space\fi MR }
\providecommand{\MRhref}[2]{%
  \href{http://www.ams.org/mathscinet-getitem?mr=#1}{#2}
}
\providecommand{\href}[2]{#2}

\appendix
\section{Appendix}\label{appendix}

In this appendix, we shall prove the lemmas stated in \cref{sec:reg} for which we could not find a proof in the literature. Their proofs however are standard and not difficult.
\begin{proof}[Proof of \cref{lem:robust}]
  Let $U_i = (V_i \setminus X_i) \cup Y_i$ for $ i \in \{1,2\}$. We will show that $ (U_1,U_2) $ is $(8\e,d - 8\e,\delta/2)$-super-regular.
   Let now $Z_i \subseteq U_i$ with $|Z_i| \geq 8 \e |U_i|$, and let $Z_i' = Z_i \setminus Y_i$ and $Z_i'' = Z_i \cap Y_i$ for $i \in \{1,2\}$.
  Note that we have
  \begin{align}
    \label{eq:help1}
    |Z_i| &\geq 8\e|U_i| \geq \e |V_i|,\\
    \label{eq:help2}
    |Z_i''| &\leq |Y_i| \leq \e^2 |V_i| \overset{\eqref{eq:help1}}{\leq} \e |Z_i| \text{ and}\\
    \label{eq:help3}
    |Z_i'| &= |Z_i| - |Z_i''| \overset{\eqref{eq:help2}}{\geq} (1-\e)|Z_i|
  \end{align}
  for both $i \in \{1,2\}$.
  We therefore have
  \[e(Z_1,Z_2) \leq  e(Z_1',Z_2') + e(Z_1'',Z_2) + e(Z_1,Z_2'') \overset{\eqref{eq:help2}}{\leq} e(Z_1',Z_2') + 2\e|Z_1||Z_2|\] and thus
  \[d(Z_1,Z_2) \leq d(Z_1',Z_2') + 2\e.\]
  On the other hand, we have
  \begin{align*}
    d(Z_1,Z_2)
    & = \frac{e(Z_1,Z_2)}{|Z_1||Z_2|}
      \geq \frac{e(Z_1',Z_2')}{|Z_1'||Z_2'|} \cdot \frac{|Z_1'||Z_2'|}{|Z_1||Z_2|} \\
    & \overset{\eqref{eq:help3}}{\geq} d(Z'_1,Z'_2){(1-\e)}^2
      \geq d(Z_1',Z_2') - 2\e
  \end{align*}
  and hence $d(Z_1,Z_2) = d(Z_1',Z_2') \pm 2\e$.
  Furthermore, by $\e$-regularity of $(V_1,V_2)$, we have $d(Z_1',Z_2') = d(V_1, V_2) \pm \e$ and we conclude
  \[d(Z_1,Z_2) = d(V_1, V_2) \pm 3\e.\]
  This holds in particular for $Z_1 = U_1$ and $Z_2 = U_2$ and therefore the pair $(U_1,U_2)$ is $(8\e,d-8\e,0)$-super-regular.
  Let $u_1 \in U_1$ now. By assumption, we have $\deg(u_1,V_2) \geq \delta |V_2|$ and therefore
  \begin{align*}
    \deg(u_1,U_2) \geq \deg(u_1,V_2 \setminus X_2) &\geq (\delta - \e^2) |V_2| \\
    &\geq (\delta - \e^2)|U_2| \geq \delta/2 \cdot |U_2|.
  \end{align*}
  A similar statement is true for every $u_2 \in U_2$ finishing the proof.
\end{proof}

The following consequence of the slicing lemma will be useful when we prove \cref{lemma:regcyl,lem:regcyl-partite}.
\begin{lemma}\label{lem:super-reg}
Let $k$ be a positive integer and $d,\e >0$ with $\e \leq 1/(2k)$. If $Z = (V_1, \ldots, V_k)$ is an $\e$-regular $k$-cylinder and $d(V_i,V_j) \geq d$ for all $1 \leq i < j \leq k$, then there is some $ \gamma \leq k\e$ and sets $\tilde V_1 \subseteq V_1, \ldots, \tilde V_k \subseteq V_k$ with
$|\tilde{V}_i| = \lceil (1-\gamma) |V_i| \rceil$ for all $ i \in [k]$
so that the $k$-cylinder $\tilde Z = (\tilde V_1, \ldots, \tilde V_k)$ is $(2\e,d - k\e)$-super-regular.
\end{lemma}
\begin{proof}
For $i \not = j \in [k]$, let $A_{i,j} := \{v \in V_i : \deg(v,V_j) < (d-\e)|V_j|\}$. By definition of $\e$-regularity, we have $\card{A_{i,j}} < \e |V_i|$ for every $i \not = j \in [k]$. For each $i \in [k]$, let $A_i=\bigcup_{j \in [k]\setminus\{i\}} A_{i,j}$.
Clearly $|A_i| < (k-1)\e |V_i|$ for every $i \in [k]$, so we can add arbitrary vertices from $V_i \setminus A_i$ to $A_i$ until $|A_i| = \lfloor (k-1) \e |V_i| \rfloor$ for every $i \in [k]$.
Let now $ \tilde V_i = V_i \setminus \tilde{A}_i$ for every $i \in [k]$ and let $\tilde Z = (\tilde V_1,\ldots,\tilde V_k)$. Observe that $|\tilde{V}_i| = \lceil (1 - \gamma)|V_i| \rceil$ for all $i \in [k]$, where $\gamma = (k-1)\e$.
It follows from \cref{lem:slicing} and definition of $A_i$ that $\tilde Z$ is $(2\e, d-\e, d-k\e)$-super-regular.
\end{proof}

Given $k$ disjoint sets $V_1,\ldots,V_k$, we call a cylinder $(U_1, \ldots, U_k)$ \emph{relatively balanced} (w.r.t.\ $(V_1, \ldots, V_k)$) if there exists some $\gamma > 0 $ so that $U_i \subset V_i$ with $|U_i| = \lfloor \gamma |V_i| \rfloor$ for every $i \in [k]$.
We say that a partition $\cK$ of $V_1\times \cdots \times V_k$ is \emph{cylindrical} if each partition class is of the form $W_1 \times \cdots \times W_k$ (which we associate with the $k$-cylinder $Z=(W_1, \ldots, W_k)$) with $W_j \subseteq V_j$ for every $j \in [k]$.
Finally, we say that $\cK = \{Z_1, \ldots, Z_N\}$ is $\e$-regular if
\begin{enumerate}[(i)]
\item $\cK$ is a cylindrical partition of $V_1\times \cdots \times V_k$,
\item each $Z_i$, $i \in [k]$, is a relatively balanced w.r.t. $(V_1,\ldots,V_k)$, and
\item all but $\e |V_1|\cdots |V_k|$ of the $k$-tuples $(v_1,\ldots,v_k) \in V_1 \times \cdots \times V_k$ are in $\e$-regular cylinders.
\end{enumerate}
For technical reasons, we will allow some of the sets $V_1,\ldots,V_k$ to be empty. In this case $(A,\emptyset)$ is considered $\e$-regular for every set $A$ and $\e >0$.
If $G$ is an $r$-edge-coloured graph and $i \in [r]$, we say that a cylinder $\cK$ is $\e$-regular in colour $i$ if is $\e$-regular in $G_i$ (the graph on $V(G)$ with all edges of colour $i$).

In~\cite{Conlon2012}, Conlon and Fox used the weak regularity lemma of Duke, Lefmann and
R\"odl~\cite{Duke1995} to find a reasonably large balanced $k$-cylinder in a $k$-partite graph. In
order to prove a coloured version of Conlon and Fox's result, we will need the following coloured
version of the weak regularity lemma of Duke, Lefmann and R\"odl. Note that, like the weak
regularity lemma of Frieze and Kannan~\cite{FrKa96}, we get an exponential bound on the number of
cylinders, in contrast to the much worse tower-type bound required by Szemer\'edi's regularity lemma
(see~\cite{FoLo17}).

\begin{thm}[Duke--Lefmann--R\"odl~\cite{Duke1995}]\label{thm:weakreg}
Let $ 0 < \e < 1/2$, $k, r \in \bN$ and let $ \beta = \e^{rk^2 \e^{-5}}$. Let $G$ be an $r$-edge-coloured $k$-partite graph with parts $V_1, \ldots,V_k$.
Then there exist some $N \leq \beta^{-k}$, sets $R_1 \subseteq V_1, \ldots, R_k \subseteq V_k$ with $\card{R_i} \leq \beta^{-1}$ and a partition
$\cK = \{Z_1, \ldots, Z_N\}$ of $(V_1 \setminus R_1) \times \cdots \times (V_k \setminus R_k)$
so that $\cK$ is $\e$-regular in every colour and $V_i(Z_j) \geq \lfloor \beta |V_i| \rfloor$ for every $i \in [k]$ and $j \in [N]$.
\end{thm}

Although the original statement of Duke, Lefmann and R\"odl~\cite[Proposition~2.1]{Duke1995} does not involve the colouring and assume that sets $V_1,\ldots,V_k$ have the same size, their proof can be easily adapted to prove~\cref{thm:weakreg}.

We are now ready to prove \cref{lemma:regcyl,lem:regcyl-partite}.
\begin{proof}[Proof of \cref{lemma:regcyl}]
Let $k,r \geq 2$, $0 < \e < 1/(rk)$ and $ \gamma = \e^{r^{8rk}\e^{-5}}$. Let $n \geq 1/\gamma$ and suppose we are given an $r$-edge coloured $K_n$.
Let $\tilde{k} = r^{rk}$ and let $V_1, \ldots, V_{\tilde{k}} \subseteq [n]$ be disjoint sets of size $\lfloor n/\tilde{k} \rfloor$ and let $G$ be the $\tilde{k}$-partite subgraph of $K_n$ induced by $V_1, \ldots, V_{\tilde{k}}$ (inheriting the colouring).
Let $\tilde{\e} = \e/2$ and $\beta = {\tilde{\e}}^{r^{2rk+1} {\tilde{\e}}^{-5}}$.
We apply \cref{thm:weakreg} to get some $N \leq \beta^{-\tilde k}$, sets $R_1 \subseteq V_1,\ldots, R_{\tilde k} \subseteq V_{\tilde k}$ each of which of size at most $\beta^{-1}$ and a partition $\cK = \{Z_1, \ldots,Z_N\}$ of
$(V_1 \setminus R_i) \times \cdots \times (V_{\tilde k} \setminus R_{\tilde k})$
which is $\tilde{\e}$-regular in every colour, and with $V_i(Z_j) \geq \lfloor \beta |V_i| \rfloor \geq 2 \gamma n$ for every $i \in [\tilde k]$ and $j \in [N]$.
Note that one of the cylinders (say $Z_1$) must be $\tilde{\e}$-regular in every colour and, since $(V_1, \ldots, V_k)$ is balanced, so is $Z_1$.
We consider now the complete graph with vertex-set $\{V_1(Z_1), \ldots, V_{\tilde{k}}(Z_1)\}$ and colour every edge $V_i(Z_1)V_j(Z_1)$, $1 \leq i < j \leq \tilde{k}$, with a colour $c \in [r]$ so that the density of the pair $(V_i(Z_1),V_j(Z_1))$ in colour $c$ is at least $1/r$.
By Ramsey's theorem \cite{Ramsey1929,Erdoes1935}, there is a colour, say 1, and $k$ parts (say $V_1(Z_1), \ldots, V_k(Z_1)$) so that the cylinder $(V_1(Z_1), \ldots, V_k(Z_1))$ is $(\tilde{\e},1/r,0)$-super-regular in colour $1$.
By \cref{lem:super-reg}, there is an $(\e,1/(2r))$-super-regular balanced subcylinder $\tilde Z_1$ with parts of size at least $\gamma n$.
\end{proof}

\begin{proof}[Proof of \cref{lem:regcyl-partite}]
Let $k \geq 2$, and let $d,\e >0$ with $2k\e \leq d \leq 1$. Let $ \gamma = \e^{k^2\e^{-12}}$ and let $G$ be a $k$-partite graph with parts $V_1, \ldots, V_k$.
Let $\tilde{\e} = \e/4$ and $\beta = \tilde{\e}^{k^2 \tilde{\e}^{-5}}$.
We may assume that $|V_i| \geq 1/\gamma$ for every $i \in [k]$ (otherwise we set $U_i := \emptyset$ for all $i \in [k]$ with $|V_i| < 1/\gamma$). In particular, we have $|V_i| \geq k/(\tilde \e \beta)$ for all $i \in [k]$.

We apply \cref{thm:weakreg} (with $r = 1$) to get some $N \leq \beta^{-k}$, sets $R_1 \subseteq V_1,\ldots, R_k \subseteq V_k$, each of which of size at most $\beta^{-1}$, and an $\tilde{\e}$-regular partition
$\cK = \{Z_1, \ldots,Z_N\}$ of $(V_1 \setminus R_1) \times \cdots \times (V_k \setminus R_k)$ with $V_i(Z_j) \geq \lfloor \beta |V_i| \rfloor$ for every $i \in [k]$ and $j \in [N]$.

Note that the number of cliques of size $k$ incident to $ R = R_1 \cup \ldots \cup R_k$ is at most
\[
\sum_{i=1}^k \beta^{-1} \prod_{j \in [k] \setminus \{i\}} |V_j| \leq \tilde{\e} |V_1| \cdots |V_k|.
\]
Furthermore, since $\cK$ is $\tilde{\e}$-regular, there are at most $\tilde{\e} |V_1| \cdots |V_k|$ cliques of size $k$ in $G$ that belong to a cylinder of $\cK$ that is not $\e$-regular. Suppose that each cylinder $Z \in \cK$ has at most $(d-2\tilde{\e})|V_1(Z)| \cdots |V_k(Z)|$ cliques of size $k$. Then the number of $k$-cliques in $G$ is at most
\begin{align*}
   \tilde{\e} |V_1| \cdots |V_k| + \sum_{Z \in \cK} (d-2\tilde{\e}) |V_1(Z)| \cdots |V_k(Z)|
  \leq (d-\tilde{\e})|V_1| \cdots |V_k|,
\end{align*}
which contradicts our hypothesis over $G$. Therefore, there is a cylinder $\tilde Z$ in $\cK$ that contains at least $(d-2\tilde{\e})|V_1(\tilde Z)| \cdots |V_k(\tilde Z)|$ cliques of size $k$.
In particular, $\tilde Z$ is $(\tilde{\e},d-2\tilde\e,0)$-super-regular and relatively balanced with parts of size at least $\lfloor \beta |V_i| \rfloor$.
Finally, we apply \cref{lem:super-reg} (and possibly delete a single vertex from some parts) to get a relatively balanced $(\e,d-(k+2)\tilde{\e})$-super-regular $k$-cylinder $Z$ with parts of size at least $\tfrac{\beta}{2} |V_i| \geq \gamma |V_i|$.
This completes the proof since $(k+2)\tilde{\e} \leq k\e \leq d/2$.
\end{proof}
\end{document}